\documentclass[12pt]{article}
\usepackage[pdftex,pagebackref,letterpaper=true,colorlinks=true,
            pdfpagemode=none,urlcolor=blue,linkcolor=blue,
            citecolor=blue,pdfstartview=FitH]  {hyperref}

\usepackage{amsmath,amsfonts}
\usepackage{amssymb}
\usepackage{amsthm}
\usepackage{graphicx}
\usepackage{color}
\usepackage{mathrsfs}

\usepackage{textcomp}  

\usepackage{tocloft}

\newcommand{\half}{\textstyle{\frac{1}{2}}}

\newcommand{\isobar}{\cong}

\newcommand{\bbH}{\mathbb{H}}

\newcommand{\N}{\mathbb{N}}
\newcommand{\Z}{\mathbb{Z}}
\newcommand{\bbP}{\mathbb{P}}
\newcommand{\Q}{\mathbb{Q}}
\newcommand{\R}{\mathbb{R}}
\newcommand{\VN}{\mathbb{VN}}

\newcommand{\Lambdahatinv}{\hat{\Lambda}^{-1}}

\newcommand{\magN}{\omega}

\newcommand{\BOM}{\boldsymbol{\Omega}}

\newcommand{\No}{\boldsymbol{\mathsf{No}}}
\newcommand{\Oz}{\boldsymbol{\mathsf{Oz}}}
\newcommand{\Nn}{\boldsymbol{\mathsf{Nn}}}

\newcommand{\dcup}{\uplus}

\newcommand{\Nmag}{\mathscr{N}}
\newcommand{\Omag}{\mathscr{O}}
\newcommand{\Mscr}{\mathscr{M}}
\newcommand{\Nscr}{\mathscr{N}}
\newcommand{\Oscr}{\mathscr{O}}

\newcommand{\hash}{\#}
\newcommand{\fcard}[1]{\hash{#1}}
\renewcommand{\fcard}[1]{|#1|}

\newtheorem{definition}{Definition}

\newtheorem{theorem}{Theorem}
\newtheorem{lemma}{Lemma}

\newtheorem{corollary}{Corollary}

\let\emptyset\varnothing

\newcommand{\Alpha}{\boldsymbol{\alpha}}

\renewcommand{\ge}{\geqslant}
\renewcommand{\le}{\leqslant}
\newcommand{\Eeq}{\overset{\to}{=}}
\newcommand{\Elt}{\overset{\to}{<}}

\newcommand{\comp}{^\mathsf{C}}
\newcommand{\Comp}{^\mathbf{C}}

\long\def\COMMENT#1{}

\begin{document}

\begin{center}
\begin{LARGE}

{\bf Counting Sets with Surnatural Numbers%
 \footnote{Magnums-17.tex}
}
\end{LARGE}

\begin{large}
Peter Lynch%
\footnote{\texttt{Peter.Lynch@ucd.ie}}
 \&\ 
 Michael Mackey,%
\footnote{\texttt{Mackey@maths.ucd.ie}}
  \\
 School of Mathematics \&\ Statistics, University College Dublin.
\end{large}

\COMMENT{  
\begin{large}
Author One%
\footnote{\texttt{Author.One@University.Col}}
\marginpar{\tiny Author Names Hidden}
 \&\ 
Author Two%
 \footnote{\texttt{Author.Two@University.Col}}
  \\
Department of Mathematics, University College.
\end{large}
}  

\end{center}

\centerline{25 January, 2025}


\bigskip
\bigskip
\bigskip
\centerline{\sc Abstract}

\begin{small}
\bigskip
How many odd numbers are there? How many even numbers?  From Galileo
to Cantor, the suggestion was that there are the same number of
odd, even and natural numbers, because all three sets can be mapped
in one-one fashion to each other.  This jars with our intuition:
cardinality fails to discriminate between sets that are intuitively
of different sizes.  

The class of \emph{surreal numbers} $\No$ is
the largest possible ordered field.  In this work we define a
function, the \emph{magnum}, mapping a selection of countable sets
to a subclass of the surreals, the {surnatural numbers} $\Nn$.
Set magnums are found to be consistent with
our intuition about relative set sizes.  The
magnum of a proper subset of a set is strictly less than the magnum
of the set itself, in harmony with Euclid's axiom, ``the whole is
greater than the part''.

Two approaches are taken to specify magnums. First, they are
determined by following the genetic assignment of magnums in the way
the surreal numbers themselves are defined. Second, the domain of the
counting sequence, which is defined for every countable set,
is extended, to evaluate it on the surnatural numbers.
The two methods are shown to be consistent. 

For a subset $A$ of $\N$, the magnum is defined as the value at $\omega$ 
of the extended counting function of $A$. Larger sets are partitioned
into finite components and a more general definition of magnums is presented.
Several theorems concerning the properties of magnums are proved,
and are employed to evaluate the magnums of a range of interesting countable sets.

The relativity of the magnum function is discussed and a number of examples
illustrate how its value depends on the choice and ordering of the reference set.
In particular, we show how the rational numbers may be ordered in such a way that
all unit rational intervals have equal magnums.

\end{small}


\bigskip
\bigskip
\bigskip
\setcounter{tocdepth}{2}
\tableofcontents


\newpage


\section{Introduction}
\label{sec:Intro}

Our objective is to define a number $m(A)$ for any countable set
$A$ that corresponds to our intuition about the
size or magnitude of $A$. We will call $m(A)$ the \emph{magnum}
of $A$.%
\footnote{As `magnum' is a contraction of `magnitude number',
we write the plural as magnums, not magna.}
For this, we require quantities that are beyond the real numbers
and more refined than the ordinal numbers. The surnatural numbers,
a subclass of the class of surreal numbers, are suitable for 
our purposes.

The surreal numbers, denoted by $\No$, were discovered by John~H.~Conway
around 1970. They form the largest possible ordered field, with all the
basic arithmetical operations, and sensible arithmetic
can be carried out over $\No$.  Moreover, quantities like $\omega/2$ and
$\sqrt{\omega}$ are meaningful (where $\omega$, the smallest infinite ordinal,
is the ``first-born'' infinite surreal number). 

A subclass of $\No$ called the omnific integers $\Oz$, 
plays the role of integers in the surreal context. We define the 
\emph{surnatural numbers} $\Nn$ to be the class of non-negative
omnific integers. The \emph{magnum} of a countable set $A$
is defined to be a surnatural number.
If $A$ is a finite set, then $m(A)$ is just the cardinality of $A$.
The magnum of $\mathbb{N}$ is $\omega$; the sets of even and odd
natural numbers both have magnums $\omega/2$. 
We find the magnums of a wide range of other  sets which are
consistent with our intuition about the relative sizes of these sets.

\subsubsection*{Background}

Galileo observed that the natural numbers and their squares could
be matched one-to-one, with each $n$ paired with its square $n^2$
and each perfect square with its root.  This caused him to doubt
the meaning of size for infinite sets.  Bernard Bolzano recognized the
crucial property of an infinite set, that it may be in one-to-one
correspondence with a proper part of itself.  However, he made clear
in his book \emph{Paradoxes of the Infinite} \cite{Bol15} that such a
correspondence is insufficient reason to regard two sets as being the
same size.  Dedekind later identified such a bijection as the defining
characteristic of an infinite set.

Cantor defined transfinite numbers and showed how every set has a
cardinal number.  However, cardinality is a \emph{blunt instrument}:
the natural numbers, rationals and algebraic numbers all have the
same cardinality, so $\aleph_0$ fails to discriminate between them.
Moreover, the arithmetic of the cardinals is awkward at best: for
example, $\aleph_0+1=\aleph_0$ and $\aleph_0^2=\aleph_0$.  

Cantor also showed that every well-ordered set has an ordinal number. 
For infinite sets, there are many possible orderings, so that 
\emph{a given set may have more than one ordinal assigned to it.}
The ordinals have the unattractive property of being non-commutative
for both addition and multiplication.  Thus, neither cardinals nor
ordinals provide a good basis for a calculus of transfinite numbers.

All infinite subsets of the natural numbers $\mathbb{N}$ have the
same cardinality. However, intuition strongly indicates that a
proper subset of a set should be ``smaller than'' the set itself.
This idea is generally called The Euclidean Principle,%
\footnote{The Euclidean Principle is also known as Aristotle's Principle
and as the Part-Whole Principle.}
as it accords with the fifth Common Notion as usually listed
in Euclid's \emph{Elements}
\cite[pg.~232]{He08} \cite{Risi20}.

In a series of publications starting with \cite{BeNa03} and culminating
in a masterful treatise, \emph{How to Measure the Infinite} \cite{BeNa19},
Vieri Benci and Mauro Di Nasso 
defined the concept of \emph{numerosity} as a generalization of
finite cardinality that also applies to infinite sets. Like magnums,
numerosity respects the Euclidean Principle, that ``the whole is
greater than the part'': the numerosity of a proper subset is
strictly smaller than that of the whole set.

There are similarities,
but also striking differences between the numerosities defined 
in \cite{BeNa19} and the set magnums defined here.  In particular, the
codomains are different: numerosities take values in the hyperreal
numbers whereas magnums are surnatural numbers.  Numerosity Theory is an
application of Alpha Theory, a simplified formulation of non-standard
analysis ($\Alpha$ is a ``new number'', the architypal infinite hyperreal
number, and the counterpart of the surreal number $\omega$). 
Numerosity Theory considers \emph{labelled sets}, which are
partitions of infinite sets into finite components. They are closely
related to the fenestrations defined in this paper. 

A model for Numerosity Theory is an equivalence relation on
$\mathbb{R}^{*}=\mathbb{R}^\mathbb{N}/\mathcal{U}$,
where $\mathbb{R}^{*}$ is the set of hyperreal numbers and
$\mathcal{U}$ is a free ultrafilter on $\mathbb{N}$
\cite[p.~221]{BeNa19}. There is an arbitrary aspect to the choice of
ultrafilter, which affects the numerosity value. For example, the
numerosities $\mathfrak{n}(2\mathbb{N})$ and $\mathfrak{n}(2\mathbb{N}-1)$
may or may not be equal; this depends on the parity of the hyperreal number
$\Alpha$ which, in turn, depends on the choice of $\mathcal{U}$.
For the surreals, a family $\BOM$ of `omega sets' that contains the set of
even numbers arises naturally. In consequence, the surreal number $\omega$
is unequivocally even, indeed divisible by any integer, and no ambiguity exists.


A simpler theory of the sizes of infinite sets that
respects the Euclidean Principle was presented by Trlifajov\'{a}
\cite{Trl24}.  This theory is strongly influenced by the work of
Bernard Bolzano, especially as presented in his book \cite{Bol15}
(see also \cite{Trl18}).
The method of Trlifajov\'{a} is elementary: no use is made of
ultrafilters, and the transfer principle is not needed.  The concept
of \emph{canonical arrangements} of sets in \cite{Trl24} is closely related
to the labelled sets of \cite{BeNa19}. 
The results are compared to Benci's numerosities. They agree for the most
part although there are striking differences, most notably the size of the
set of rational numbers.

Several other ways of expressing the sizes of infinite sets have been proposed.
In a recent review, Wenmackers (2024) compared six methods, including
cardinality, natural density and numerosity.  She described their merits
and shortcomings, and discussed how, for each method, some desirable
property must be abandoned. She stressed the non-unique and non-constructive
aspects of sizes that are assigned to subsets of $\N$ that are both infinite
and co-infinite.  For example, the methods of \cite{BeNa19} and \cite{Trl24}
do not uniquely determine the relative sizes of the odd and even natural numbers.


\subsubsection*{Transfer and Extension}

A preliminary development of the theory presented here
\cite{LyMa23} depended upon a \emph{Transfer Principle},
which posits that properties that hold for real numbers remain valid
when transferred to the surreal domain.  According to the
\emph{Encyclopedia of Mathematics} \cite{EoM},
\begin{quote}
\begin{small}
A Transfer Principle
allows us to extend assertions from one algebraic system $\mathit{T1}$
to another such system $\mathit{T2}$. Every elementary statement,
such as a formula, equation or inequality in the first-order language
of $\mathit{T1}$ is true for the extension $\mathit{T2}$.
\end{small}
\end{quote}
The Transfer Principle is a central premise in non-standard analysis
\cite{Eh22,Kei76,Kei13}.  However, it cannot be assumed in the
context of the surreal numbers. A simple counter-example illustrates
the problem: the real number $\sqrt{2}$ is irrational --- it is not
a ratio of two integers; however, the surreal number $\sqrt{2}$ is
the ratio of $\sqrt{2}\omega$ and $\omega$, both of which are omnific
integers. The theory presented in this paper does not employ a
transfer principle but depends only on a much weaker assumption,
the Axiom of Extension which, in many cases of interest, is obviously
valid.


\COMMENT{  
\subsection*{In Conclusions}

We note that Benci et al.~\cite{BeNa19} define numerosities on a family
$\mathfrak{U}=\bigcup_{n\in\mathbb{N}}(\bigcup_{k=1}^n \mathbb{N}^k)$,
which is a \emph{universe}, closed under subsets, disjoint unions and
Cartesian products.
}  

\subsubsection*{Desiderata for the Magnum Function}
\label{sec:Desiderata}

We should like to assign a ``size'' to any subset $A$ of $\mathbb{N}$
that accords with our intuition about its magnitude.
We denote the size of a set $A$ by $m(A)$,
called the {``magnitude number''} or \emph{magnum} of $A$.
We list below some desirable properties of the magnum function $m(A)$.
\begin{enumerate}
\item For a finite subset $A$, we require $m(A) = \fcard{A}$, the number of elements in $A$.
\item A \emph{proper} subset%
\footnote{Throughout this paper we denote subsets by $\subseteq$ and
proper subsets by $\subset$.}
$A$ of set $B$ should have a \emph{smaller} magnum than $B$:
$$
A \subset B \implies m(A) < m(B) \,.
$$
\item Addition of a new element $b\notin A$ should increase the magnum by one:
$$
m(A \dcup \{b\}) = m(A) + 1 
$$
\item Removal of an element $a\in A$ from a (non-void) set should reduce the magnum by one:
$$
m(A \setminus \{a\}) = m(A) - 1 
$$
\item The magnums of the union of two disjoint sets should add:
if $A \cap B = \varnothing$, then
$$
m( A \dcup B ) = m(A) + m(B)
$$
\item For the complete set of natural numbers, $m(\mathbb{N}) = \magN$.
\item
The odd and even numbers should behave \emph{sensibly}:
\begin{eqnarray*}
2\mathbb{N} = \{ 2, 4, 6, \dots \} &\implies& m(2\mathbb{N}) = \half \magN \\
2\mathbb{N}-1 = \{ 1, 3, 5, \dots \} &\implies& m(2\mathbb{N}-1) = \half \magN \,,
\end{eqnarray*}
ensuring that $m((2\mathbb{N})\dcup(2\mathbb{N}-1)) = 
m(2\mathbb{N}) + m(2\mathbb{N}-1) = \magN$.
\item Similar requirements apply for arithmetic sequences. For example,
\begin{equation}
k\mathbb{N} = \{ k, 2k, 3k, \dots \} \implies  m(k\mathbb{N}) = \frac{\omega}{k}.
\end{equation}
\item The size of the set of perfect squares should be `reasonable:'
$$
\mathbb{N}^{(2)} = \{ 1, 4, 9, 16, \dots \} \implies m(\mathbb{N}^{(2)}) = \sqrt{\omega}
$$
\item The magnum should be consistent with standard results of number theory. For example,
by the prime number theorem, the prime counting function $\pi(n)$ satisfies
$\pi(n) \approx n/\log n$. We expect the magnum of the set $\mathbb{P}$ of
prime numbers to satisfy
$$
m(\bbP) \approx \omega/\log\omega .
$$
\end{enumerate}
The objective is to define the magnum in such a manner that the above properties hold.


\subsubsection*{Outline of Contents}
\label{sec:outline}

A preliminary study of magnums, which appeared in \cite{Lyn19},
was developed and extended in \cite{LyMa23}.
The present paper supersedes those two studies.

Part~I developes the theory of magnums for subsets of $\N$.
The genetic development of magnums, following the evolution of
the surreal numbers day by day, is presented in \S\ref{sec:gendef}.
In \S\ref{sec:density}, the counting sequence and density sequence are introduced.
We explore the value of natural density in defining the sizes
of sets and discover its limitations.  In \S\ref{sec:theorems} 
several theorems concerning counting sequences are
stated and proved. The concept of isobaric equivalence
is introduced.
In \S\ref{sec:extension},
We examine how the domain of functions can be extended
from the natural to the surnatural numbers, and we
specify the Axiom of Extension.
 
In \S\ref{sec:DefMag}, the magnum of a subset $A$ of $\N$ is defined as
the value of the extended counting function of $A$ at $\omega$. 
Several important theorems on magnums are stated and proved in \S\ref{sec:M-theorems}.
In \S\ref{sec:Equiv} we show that the definition of magnums using the counting sequence
is consistent with the genetic approach. 
Fenestrations of sets are defined and the General Isobary Theorem is proved.

In Part~II, sets larger than or beyond $\N$ are considered.
In \S\ref{sec:largersets}, magnums are defined for sets $A$
such that $A\setminus\N \ne \emptyset$.
We generalize the indicator and counting point-functions to the 
set-functions $X(n)$ and $K(n)$ and use them to define magnums for
larger sets. We also illustrate how magnums can be defined
relative to reference sets other than $\N$. 

The set of rational numbers is considered in \S\ref{sec:rationals}
and we define an ordering of $\Q$ such that the magnum of a bounded set
of rational numbers is invariant under translation (a property regarded by
Bolzano as essential).
With this ordering, all unit rational intervals have equal magnums.
Theorems are applied in \S\ref{sec:Applics}
to deduce the magnums of a broad selection of sets.

Concluding remarks and observations about the prospects and opportunities
for future developments of the theory are presented in \S10.
Two Appendices are included: the first reviews the 
surreal numbers and surnatural numbers, the second
contains a calendar of magnums for selected sets.

\addcontentsline{toc}{part}
{Part~I: Subsets of $\N$} 
\part*{Part~I: Subsets of $\N$}
\label{sec:part1}
\section{Genetic Development of Magnums}
\label{sec:gendef}


The ultimate goal is to define magnums for all countable sets. In this
section, the focus is on $\mathscr{P}(\N)$, the power set of the natural
numbers.  We show how magnums of subsets of natural numbers may be 
generated using the \emph{genetic approach}: given a set $A\in\mathscr{P}(\mathbb{N})$,
we seek two sets of magnums, $L_A$ and $R_A$, such that
the surreal number $\{L_A\ |\ R_A\}$ is the magnum of $A$.
This will be done in a manner analogous to that in which the surreal
numbers themselves are defined [see Appendix~1].

The values of set magnums will be in the semi-ring of surnatural numbers $\Nn$.
The surnatural numbers are the non-negative omnific integers,
$\Nn=\Oz^{+}\cup\{0\}$.  The omnific integers and surnatural numbers,
first introduced in \cite[Chap.~5]{ONAG}, are briefly described in Appendix~1.
The immediate challenge is to construct a function $m$ from the power set
of the natural numbers to the {surnatural numbers},
$$
m : \mathscr{P}(\mathbb{N}) \to \Nn \,,
$$
that has the properties --- or desiderata --- outlined in the Introduction.
In particular, we require the Euclidean Principle to hold,
$m(A)= \fcard{A}$ for finite $A$, finite additivity of the mapping $m$,
and equal magnums for isobaric sets.
Two sets $A_1$ and $A_2$ are \emph{isobaric} if, for a simple partition of $\N$
into disjoint, finite sets $(W_n)_n$, all of equal size $L$, we have
$|A_1\cap W_n|=|A_2\cap W_n|$ for all $n$; see Definition~\ref{def:regfen}
in \S\ref{sec:isobar}, where we prove that isobary is an equivalence relation
and \S\ref{sec:GIT}, where we show that isobaric sets have equal magnums.


\subsection{Pre-magnums and the Magnum Form}
\label{sec:premagnums}

We seek a general expression for the magnum of a set $A$ in the form
\begin{equation}
m(A) = \{ m(B) : B \in \mathscr{M}, B \subset A \ | \ m(C) : C \in \mathscr{M}, A \subset C \} \,,
\label{eq:fullmagform}
\end{equation}
where $\mathscr{M}$ is the family of all sets whose magnums are to be defined.%
\footnote{We would hope that $\mathscr{M}$ comprises all the countable sets. 
However, the full extent of the family $\mathscr{M}$ remains to be determined.}
We note that the form (\ref{eq:fullmagform}) automatically guarantees that the
Euclidean Principle applies, because $B\subset A\subset C$ implies $m(B) < m(A) < m(C)$.
However, we cannot use (\ref{eq:fullmagform}) to define or evaluate $m(A)$, since it
requires knowledge of the magnums of all the other sets in $\mathscr{M}$. 
We need to construct the magnums in a \emph{recursive} or incremental fashion.



In the genetic approach, when defining the magnum of a set on Day~$\alpha$,
we use the magnums of `old' sets to generate the magnums of `new' sets.
We express the assignment of a magnum to a set $A$ on Day~$\alpha$ by saying that
$A$ is \emph{magnumbered} on Day~$\alpha$.
For each ordinal number $\alpha$, Conway \cite[pg.~29]{ONAG} defines the sets
$M_\alpha$, $N_\alpha$ and $O_\alpha$:
\begin{itemize}
\item $M_\alpha$ is the set of numbers born on or before Day~$\alpha$ (Made numbers),
\item $N_\alpha$ is the set of numbers born first on Day~$\alpha$ (New numbers), and 
\item $O_\alpha$ is the set of numbers born before Day~$\alpha$ (Old numbers).
\end{itemize}
Note that $M_\alpha = N_\alpha \uplus\, O_\alpha = O_{\alpha+1}$.
In a similar manner, we define three collections of sets:
\begin{itemize}
\item $\Mscr_\alpha$ are the [\emph{made}] sets that are magnumbered on or before Day~$\alpha$,
\item $\Nscr_\alpha$ are the [\emph{new}] sets that are magnumbered on Day~$\alpha$, and
\item $\Oscr_\alpha$ are the [\emph{old}] sets that are magnumbered before Day~$\alpha$.
\end{itemize}
The last two families combine to give the first:
$\Mscr_\alpha$ = $\Nscr_\alpha \uplus\, \Oscr_\alpha = \Oscr_{\alpha+1}$.  



Since, for a given set $A$, we do not know on which day its magnum is determined,
we introduce a sequence of numbers, one for each ordinal $\gamma$, defined by
\begin{equation}
m_\gamma(A) = \{ m(B) : B \in \Oscr_\gamma, B \subset A \ | \ m(C) : C \in \Oscr_\gamma, A \subset C \} \,.
\label{eq:premag}
\end{equation}
In this recursive construction, the proper subsets $B$ and
supersets $C$ of $A$ range over all sets whose magnums have been defined \emph{prior to} Day~$\gamma$.
As the day number increases, the value of the pre-magnum $m_\gamma(A)$ of $A$ may change;
$m_\gamma(A)$ is the nearest we can get to $m(A)$ on Day~$\gamma$ \cite{Sim17}.  
When a stage $\gamma=\alpha$ is reached where $m_\alpha(A)$ cannot undergo further changes,
we assign the magnum of $A$ to be
\begin{equation}
m(A) = m_\alpha(A)
\label{eq:magnumdef}
\end{equation}
and call $\alpha$ the birthday of $m(A)$.

The specific choices of sets to which magnums are tentatively assigned can
sometimes be justified by set-theoretic manipulations that confirm their properties.
In other cases, the assignments are provisional, requiring confirmation which is
provided in subsequent sections of this paper. The assignments via genetic forms
are shown to be consistent with the definition of set magnums introduced
in \S\ref{sec:DefMag} (see \S\ref{sec:Equiv}).

\renewcommand{\emptyset}{\varnothing}


\subsection{Day-by-day Construction of Magnums}
\label{sec:daybyday}


We construct magnums \emph{one day at a time}, as for the surreal numbers themselves.
Our strategy is to start with the least element of $\Nn$, that is $0 = \{\varnothing |\varnothing \}$ 
and proceed from each number to its successor(s), considering magnums in the same
order in which the underlying surnatural numbers are defined.
At each stage, we consider candidate sets that might have these numbers as their magnums.
In some cases, this is obvious; in other cases, further investigation is needed to
identify the appropriate set or sets. Proceeding genetically, we obtain magnums
for all the finite subsets of $\mathbb{N}$ by
Day~$\omega$, for all the cofinite subsets of $\mathbb{N}$ by Day~$2\omega$,
and many more sets by Day~$\omega^2$.


\textbf{Day 0:}
At this point, no magnums have been defined so that, for \emph{any}
set $A$, we must have $L_A = R_A = \emptyset$.
Thus, $m(A) = \{\ \emptyset\ |\ \emptyset\ \} = 0$.
Unless $A$ is empty, it has some proper subset $B$, and
the magnum of $B$ will be negative, which is outside the codomain $\Nn$.
Thus, the only possibility is $A=\emptyset$, so that $m(\emptyset)= 0$.

\textbf{Day 1:}
The only magnum number available for Day~1 is $0$, so the only possible
new value is $\{\ 0 \ |\ \emptyset\ \} = 1$.
For any set $A$ other than a singleton, there will exist
a proper subset $B$ with $\emptyset \subset B \subset A$.
We conclude that, for all $a\in\mathbb{N},\  m(\{a\}) = 1$.

\textbf{Days 2, 3, \dots :}
The only magnums available for Day~$n$ are $\{0, 1, \dots , n-1 \}$, so
the only possible new value is $\{0, 1, \dots , n-1 \ |\ \emptyset\ \} = n$.
Since one of our desiderata is that, for finite sets, $m(A) = \fcard{A}$,
it is natural to consider singletons on Day~1, doubletons on Day~2, tripletons
on Day~3, and so on. Thus we will find that
$$
m(\{a\}) = 1\,, \quad  m(\{a_1,a_2\}) = 2\,, \quad \dots \quad m(\{a_1,a_2,\dots,a_n\}) = n\,, \quad \cdots \,.
$$
We easily see that all of these cases concur with the form (\ref{eq:premag}).
In this way, we define the magnums for all finite subsets of $\mathbb{N}$.
Moreover, for all such sets, $m(A) = \fcard{A}$.%

\textbf{Day $\boldsymbol{\omega}$:}
The obvious choice for Day~$\omega$ is $A = \mathbb{N}$,
when we obtain our first infinite magnum:
\begin{equation}
m(\mathbb{N}) = \{ 1, 2, 3, \dots \ |\ \emptyset\ \} = \omega \,.
\label{eq:m-omega}
\end{equation}
Since we are (for the present) confining attention to the power-set of the
natural numbers, $\mathbb{N}$ is the only set that gets a magnum on Day~$\omega$.
Later, we will establish results such as $m(2\N\sqcup 2\N) = \omega$, where
the disjoint union $A\sqcup B$ is defined in \S\ref{sec:largersets}.

In the construction of the surreal numbers, all dyadic rationals
are born before Day~$\omega$ and all remaining `real' surreal numbers
emerge on that day. However, the only surnatural number to emerge
on Day~$\omega$ is $\omega$.
From this point onwards, $\omega$ may appear on the right side of the magnum form.

\textbf{Days $\boldsymbol{\omega}+1$ to {$2\boldsymbol{\omega}$:}}
We now ``work downwards'' from $\omega$.%
\footnote{When we consider sets beyond $\mathscr{P}(\mathbb{N})$,
the number $\omega+1$ will be available as a magnum.}
For $A = \mathbb{N}\setminus \{a\}$, the only proper superset is $\mathbb{N}$,
so, for any $a$,  $m(A) = \{\mathbb{N}\ |\ \omega\} = \omega-1$.
More generally, if we remove $n$ elements from $\mathbb{N}$, a set of magnum $\omega-n$
results.  We have now defined the magnums of all finite and all cofinite sets
in $\mathscr{P}(\mathbb{N})$.

\textbf{Day {$2\boldsymbol{\omega}$:}}
On Day~$2\omega$, we remove an infinite number of elements of $\mathbb{N}$,
so, the elements of $\{ \omega,\omega-1,\omega-2,\dots \}$ will be in $R_A$.
There are several possibilities:
\begin{itemize}
\item If we have removed \emph{all} elements of $\mathbb{N}$,
the magnum must be $0$, the magnum of $\emptyset$.
\item If we have left only a finite number, $k$, of elements, the magnum is
$k$.
\item If we have removed an infinite number and also left an infinite number,
then both $A$ and $\mathbb{N}\setminus A$ are infinite, and the new magnum is
$$
m(A) = \{\ \mathbb{N}\ |\ \omega,\omega-1,\omega-2,\dots \}  = \tfrac{\omega}{2} \,.
$$
\end{itemize}
But what sets have magnum $\omega/2$?
One obvious candidate is $2\N$, with
$$
m(2\mathbb{N}) = m(\{2,4,6, \dots \})
= \{\ \mathbb{N}\ |\ \omega,\omega-1,\omega-2,\dots \}  = \tfrac{\omega}{2} \,.
$$
The set of odd natural numbers, the complement of $2\mathbb{N}$, 
also has magnum ${\omega}/{2}$, so $m(2\mathbb{N}-1) = \omega/2$,
as have all sets isobaric to $2\N$.
These assignments will be confirmed below.


\textbf{Day {$2\boldsymbol{\omega}+1$ to $3\boldsymbol{\omega}$:}}
We now have a new limit magnum, $\omega/2$, which may appear on either the left or right side of the
magnum form. The immediate successors of $\omega/2$ are $\frac{\omega}{2} - 1$ and
$\frac{\omega}{2} + 1$.  The first provides the magnum for sets formed by 
removing a single even number from $2\mathbb{N}$,
for example, $2\mathbb{N}\setminus\{2\}$.
The second provides the magnum for sets formed by 
adding a single odd number to $2\mathbb{N}$,
for example, $2\mathbb{N}\uplus\{1\}$.
Both $\frac{\omega}{2}-1$ and $\frac{\omega}{2}+1$ are surnatural numbers
but only one successor of each is such.
We can form sets with magnums for these: for example, at Day~$2\omega+n$ we have
$$
m(2\mathbb{N}\setminus \{e_1, e_2, \dots,e_n\}) = \tfrac{\omega}{2}-n \,, \qquad
m(2\mathbb{N}\uplus \{o_1, o_2, \dots,o_n \}) = \tfrac{\omega}{2}+n \,,
$$
where $\{e_1, e_2, \dots,e_n\}$ and 
$\{o_1, o_2, \dots,o_n\}$ are, respectively, sets of $n$ even and odd integers.

\textbf{Day {$3\boldsymbol{\omega}$:}}
Following the pattern of creation of surreal numbers, where the dyadic fractions
all appear before any other numbers, we find a new number on Day~$3\omega$:
$$
\{ 1, 2, 3, \dots \ |\ \tfrac{\omega}{2},\tfrac{\omega}{2}-1,\tfrac{\omega}{2}-2,\dots \ \}
= \tfrac{\omega}{4} \,.
$$
There are many sets that may have magnum $\omega/4$, but the canonical choice
is $4\mathbb{N} = \{4, 8, 12, \dots\ \}$.

\textbf{Day {$3\boldsymbol{\omega}+1$ to $\boldsymbol{\omega}^2$:}}
We continue the process, with a new limit ordinal entering each
$\omega$ days until, on Day~$\omega^2$, we have all surreal numbers
of the form $ ((2m+1)/2^k)\omega \pm \ell$, which may be used on
either the left or right side of the magnum form.
We can form sets with magnums for these: for example, at Day~$4\omega+n$ we have
$$
m(4\mathbb{N}\setminus \{2e_1, 2e_2, \dots,2e_n\}) = \tfrac{\omega}{4}-n \,, \qquad
m(4\mathbb{N}\uplus \{o_1, o_2, \dots,o_n \}) = \tfrac{\omega}{4}+n \,,
$$
where $\{e_1, e_2, \dots,e_n\}$ and $\{o_1, o_2, \dots,o_n\}$ are, respectively,
sets of $n$ even and odd integers.

\textbf{Day {$\boldsymbol{\omega}^2$:}}
On Day~$\omega^2$ we have $m(k\N)=\frac{\omega}{k}$ for all $k\in\N$, so $\Oscr_{\omega^2}$
contains these sets and their isobaric equivalents.
We can also form
$$
\{ 1, 2, 3, \dots \ |\ \tfrac{\omega}{2},\tfrac{\omega}{4},\tfrac{\omega}{8}\dots \ \} = \sqrt{\omega} \,.
$$
We expect that the set of squares of the natural numbers,
$\N^{(2)} = \{1, 4, 9, 16, \dots \}$, has magnum
$\sqrt{\omega}$; this will be confirmed in {\S\ref{sec:Applics}.

In Appendix~2, a Calendar of some of the principle sets whose
magnums are defined on or before Day~$\omega^2$ is presented.

\textbf{Days beyond $\boldsymbol{\omega}^2$:}
The process may be continued with numbers like $\omega^{1/4}$,
$\omega^{1/n}$ and indeed $\omega^{1/\omega}$ appearing. More
generally, surnatural numbers of the form $\sum r_y\cdot\omega^{y}$
with $y\in[0,1]$ arise; the construction process never ceases.
Ultimately, we have assigned magnums to a large family $\mathscr{M}$ of sets.%


In \S\ref{sec:DefMag} we present a definition of the magnum of a set $A$ in
terms of the counting function of $A$. This will permit us to confirm the
choices made above.



\section{Natural Density and Set Size}
\label{sec:density}

In this section we examine the suitability of the natural density
of a set $A$ for defining a magnum for that set. In some cases, this
yields the expected value but in others it does not. Moreover, there
are many sets for which the density is undefined.  In \S\ref{sec:extension}
we will consider a more powerful and general method, using the counting sequence.


\subsubsection*{The Defining Sequence}

We assume that the elements of $A\subseteq\mathbb{N}$ are listed
in order of increasing size, with the \emph{defining sequence}
$a_A: k\mapsto a_k$.  The graph of $a_A$ is the set of points
$\{(1,a_A(1)),\dots,(k,a_A(k)),$ $\dots\}$ and, since $a_A$
is strictly increasing, the function $a_A$ has an inverse,
defined by $a_A^{-1}: a_k\mapsto k$, with
graph $\{(a_A(1),1),\dots,(a_A(k),k),\dots\}$.

The sequence $a_A^{-1}$ is defined on the set $A$.
It may be extended to the domain $\mathbb{N}$,
either by linear interpolation or by using a monotone algebraic
expression if that is available.



\subsubsection*{The Counting Sequence}

\newcommand{\areal}{\alpha}

Let $A$ be a subset of $\N$ with monotone increasing defining sequence
$a_A : k \mapsto a_k$.  We define the counting sequence $\kappa_A(n)$
to be the number of elements of $A$ less than or equal to $n$, so that
$\kappa_A(n) = | A \cap I_n |$ where $I_n = \{ 1, 2, \dots , n \}$.
There is an intimate connection between
$\kappa_A$ and the defining sequence $a_A$, such that
$\kappa_A(n)$ is the integer part of the inverse of $a_A$
evaluated at $n$.
\begin{theorem}
Let $A = \{a_n:n\in\N \}$ be a subset of $\N$ with a monotone increasing
defining sequence $a_A : k \mapsto a_k$. We interpolate $a_A:\N\to\N$ to
a strictly increasing continuous function $\areal:\R^{+}\to\R^{+}$ on the
positive real numbers.  Then, the counting sequence of the set $A$ is
\begin{equation}
\kappa_A(n) = \lfloor \areal^{-1}(n) \rfloor .
\label{eq:KapAlp}
\end{equation}
\label{th:16.24}
\end{theorem}
\begin{proof}
Since $\areal$ is a strictly increasing function on $\R^{+}$, its inverse
$\areal^{-1}$ is also strictly increasing on $\R^{+}$.  We can compute the
integer part of $\areal^{-1}(n)$ starting from the counting function
$\kappa_A(n) = \{ \areal(1), \areal(2), \dots , \areal(k) \}$,
where $k$ is the largest value in $\N$ for which $\areal(k)\le n$. Thus, 
$$
\kappa_A(n) = \max\{\ell\in\N : \areal(\ell)\le n \} 
            = \max\{\ell\in\N : \ell\le \areal^{-1}(n) \} 
            = \lfloor \areal^{-1}(n) \rfloor .
$$
The floor operator ensures that $\kappa_A(n)$ is a natural number for all 
$n\in\mathbb{N}$. It also ensures that the value of $\kappa_A(n)$ is independent of
the manner in which $\areal:\R^{+}\to\R^{+}$ is interpolated continuously between
values of $a_A:\N\to\N$.
\end{proof}
Results equivalent to (\ref{eq:KapAlp}) are also proved in
\cite[ Prop.~16.24]{BeNa19} and \cite[Th.~6]{Trl24}.

If $A$ is finite, it is convenient to extend $\kappa_A$ to $\mathbb{N}$.
The obvious way to do this is to define $\kappa_A(n) = N$ for all
$n\ge a_A(N)$ where $N$ is the cardinality of $A$. We will see below
that this leads to the magnum $m(A) = N$.  
If $A$ is infinite, then $\kappa_A(n)$ increases without bound.


\subsubsection*{The Density Sequence}

The \emph{natural density} of a sequence $A : \mathbb{N} \to \mathbb{N}$ is
defined, if it exists, as the (real variable) limit of the density sequence:
$$
\rho_A = \lim_{n\to\infty} \rho_A(n) \,,
$$
where $(\rho_A(n))_n$ is calculated in terms of the counting function:
\begin{equation}
\rho_A(n) = \frac{\kappa_A(n)}{n} \,.
\label{eq:defden1}
\end{equation}
An alternative definition \cite{Ten95} uses the defining function $a_A(n)$:
\begin{equation}
\boldsymbol{\varrho}_A(n) = \frac{n}{a_A(n)} \,.
\label{eq:defden2}
\end{equation}
We note that $\rho_A(n)$ has terms with denominators for every natural number,
whereas the denominators of $\boldsymbol{\varrho}_A(n)$ are confined to elements of $A$.
Thus, $\boldsymbol{\varrho}_A(n)$ is a subsequence of $\rho_A(n)$ so that, if 
$\rho_A(n)$  converges to $\rho_A$ then so does $\boldsymbol{\varrho}_A(n)$.
The connection between the two sequences is $\boldsymbol{\varrho}_A(n) = \rho_A(a_A(n))$.

\subsubsection*{A Graphical Example}

\begin{figure}[h]
\begin{center}
\includegraphics[width=0.45\textwidth]{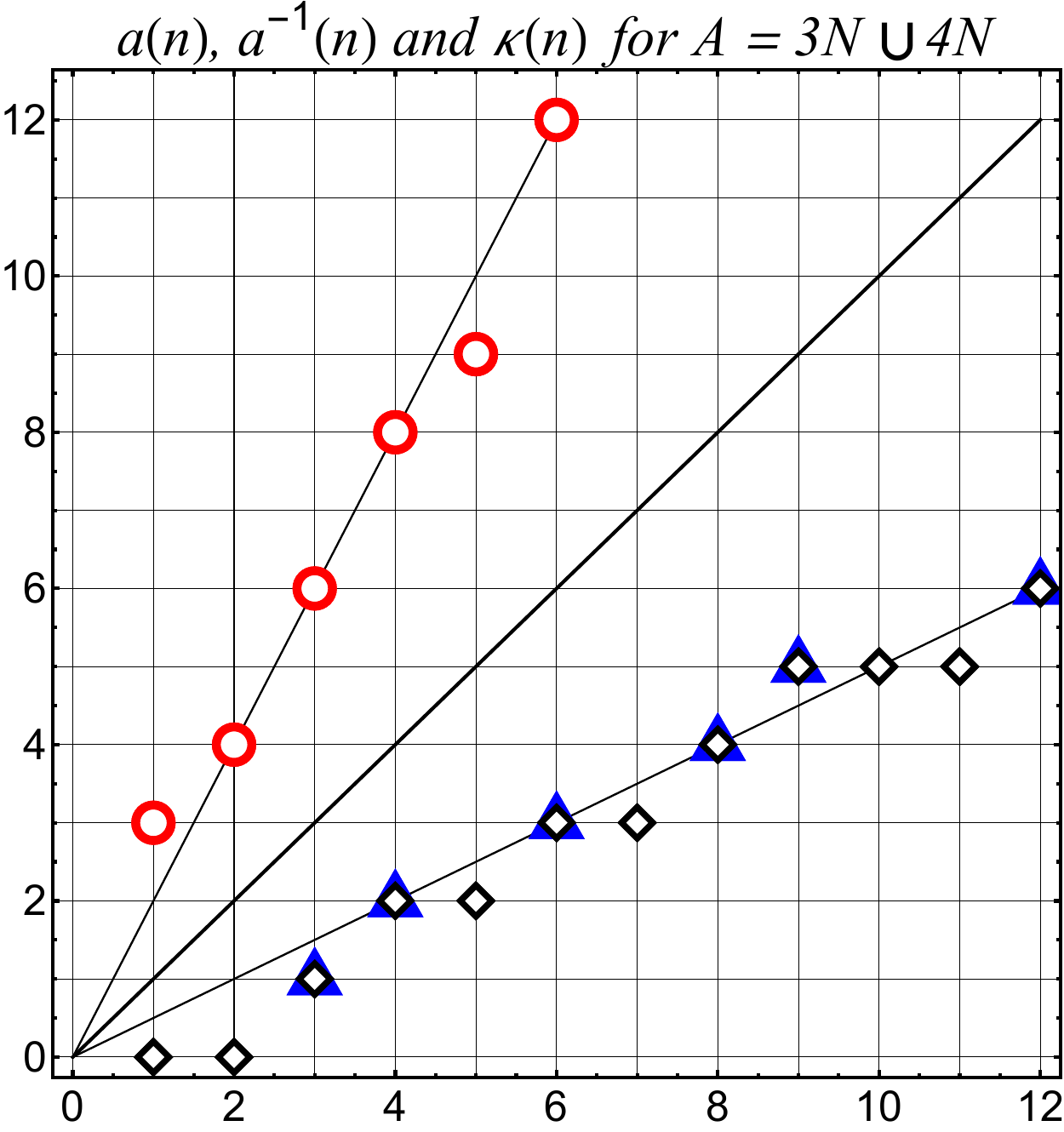}
\caption{The initial values of the defining sequence $a_A(n)$ (red circles),
         its inverse $a_A^{-1}(n)$ (blue triangles), and the counting sequence
         $\kappa_A(n)$ (black diamonds) for the set $A = 3\mathbb{N}\cup 4\mathbb{N}$.}
\label{fig:AlphaKappa}
\end{center}
\end{figure}

To illustrate the sequences defined above, we consider a simple example,
the set $A = 3\mathbb{N}\cup 4\mathbb{N}$. We will prove later that 
$m(A) = m(3\mathbb{N}) + m(4\mathbb{N}) - m(12\mathbb{N}) = \frac{\omega}{2}$.
In Fig.~\ref{fig:AlphaKappa} we show the initial values of the defining sequence
$a_A(n)$ (circles, red online), the inverse of this function $a_A^{-1}(n)$
(triangles, blue online), and the counting sequence $\kappa_A(n)$ (black diamonds) for
the set $A$.  The $n$-th term of the natural density sequence $\rho_{A}(n)$
is the slope of the line from the origin to the point $(n,\kappa_A(n))$.
The slope of the line from the origin to the point $(n,a_A(n))$
is the \emph{sparsity}, $\boldsymbol{\varsigma}_A(n) = 1/\boldsymbol{\varrho}_A(n)$.

\subsubsection*{Defining $m(A)$ by using the Density}%

For sets whose density is defined,
we can use the density to obtain a
surreal expression for $m(A)$.  In the case where $\rho_A\ne 0$,
we might attempt to define the magnum of $A$ as 
\begin{equation}
m(A) := \rho_A\cdot\omega \,.
\label{eq:rhomag}
\end{equation}
In \S\ref{sec:M-theorems} we prove the Density Theorem:
if $\rho_A$ exists, then $m(A)\approx\rho_A\omega$.
%
%
%
Some examples will illustrate the evaluation of magnums using (\ref{eq:rhomag}):

$\bullet$
For $A = \mathbb{N}$, the defining function $a_{\mathbb{N}}$
is the identity function: $a_\mathbb{N}(n) = \mathrm{id}(n) = n$.
The counting function for $\mathbb{N}$ is $\kappa_{\mathbb{N}}(n) = n$
and the density sequence is constant with $\rho_\mathbb{N}(n) = 1$,
so that its (real) limit is $\rho_\mathbb{N} = 1$.
From (\ref{eq:rhomag}) it follows that $m(\mathbb{N}) = \omega$.

$\bullet$ 
For $A = 2\mathbb{N} = \{2, 4, 6, \dots \}$ we have
$\rho_{2\mathbb{N}}(2k-1) = \frac{k-1}{2k-1}$ and $\rho_{2\mathbb{N}}(2k) = \frac{k}{2k}$ 
with limiting value $\rho_{2\mathbb{N}} = \frac{1}{2}$. Thus,
$$
m(2\mathbb{N}) = \rho_{2\mathbb{N}}\cdot\omega = \frac{\omega}{2} \,.
$$

  
$\bullet$ 
For $A = 2\mathbb{N}-1 = \{1, 3, 5, \dots \}$ we have
$\rho_{2\mathbb{N}-1}(2k-1) = \frac{k}{2k-1}$ and $\rho_{2\mathbb{N}-1}(2k) = \frac{k}{2k}$ 
with limiting value $\rho_{2\mathbb{N}-1} = \frac{1}{2}$. Thus,
$$
m(2\mathbb{N}-1) = \rho_{2\mathbb{N}-1}\cdot\omega = \frac{\omega}{2} \,.
$$

$\bullet$ 
For $A = k\mathbb{N} =\{k,2k,\dots,nk,\dots \}$ it is easy to show that $\rho_{k\mathbb{N}} = 1/k$
and it follows that $m(k\mathbb{N}) = \omega/k$.


\subsubsection*{Dependence of Size upon Ordering}
\label{sec:densorder}

The counting sequence of a set depends on the ordering of the reference set $\N$.
This was discussed in \cite{LyMa24} and a single example will suffice here.
Under the canonical ordering of the natural numbers,
$\N = \{1, 2, 3, \dots \}$, the even numbers have counting sequence
$\kappa_{2\N}(n) = (\lfloor\frac{n}{2}\rfloor)_n$ and the density is $\rho(2\N) = \frac{1}{2}$.
Now let us reorder $\N$, listing two odd numbers followed by an even one:
$$
\widetilde{\N} = \{1,3,2; 5,7,4; 9,11,5; \dots \} .
$$
Then the counting function of the even numbers is 
$\kappa_{\widetilde{2\N}}(n) = (\lfloor\frac{n}{3}\rfloor)_n$
and the density is $\rho(\widetilde{2\N}) = \frac{1}{3}$.

Since magnums are defined in terms of counting sequences, their values are
sensitive to the chosen ordering of the reference set $\N$.  Unless otherwise
stated, we will assume that $\N$ has its canonical ordering and all subsets
of natural numbers inherit their ordering from $\N$.

The ordering of larger sets, such as $\N\times\N$, will be discussed in
\S\ref{sec:largersets}.  When a set $A$ not contained in $\N$ is considered,
its magnum will be defined \emph{relative to a reference set $R$},
and its value will depend on the ordering of that set.


\subsubsection*{Difficulties Defining Magnums using Density}

The above method avoids the direct extension of $\kappa_A(n)$, for which the conditions
are, as yet, imprecise. The density sequence $\rho_A(n) = \kappa_A(n)/n$ is 
\emph{always defined}, although the limit exists only in some cases.
Moreover, the usual real-variable limit is insufficiently discriminating.
To illustrate the problems, let us consider two sets, $A=2\mathbb{N}\uplus\{1\}$
and $A = \{n^2 : n\in\mathbb{N}\}$.

$\bullet$
For $A=2\mathbb{N}\uplus\{1\}$ we find $\rho_A = \frac{1}{2}$ and $m(A)=\omega/2$.
The method fails to catch the addition of the term $a_A(1) = 1$.
We should have hoped for $m(A) = \frac{\omega}{2}+1$. 

$\bullet$
For $A = \{n^2 : n\in\mathbb{N}\}$ we have $\rho_A(n) \approx 1/\sqrt{n}$ so that $\rho_A=0$.
The limiting process has obliterated all information about the sequence, giving $m(A)=0$.
The usual limit loses information that could produce higher-order corrections to the magnum.
We really need a finer-grain description of set size.

There are several paths to explore. For example, instead of considering 
$\rho_A(n) = {\kappa_A(n)}/{n}$, we might define 
$\sigma_A(n) = {\kappa_A(n)}/{f(n)}$, choosing $f(n)$ such that 
$\sigma_A = \lim_{n\to\infty} \sigma_A(n)$ is finite and positive. 
However, for a set like
$A = \{n^2 : n\in\mathbb{N}\} \cup \{n^3 : n\in\mathbb{N}\}$,
we need $\sigma_A(n)$ to have components asymptotic to both $\sqrt{n}$ and $\sqrt[3]{n}$.

Wenmackers (2024) emphasised that all methods of specifying the sizes of infinite sets
involve compromises.  She observed that natural density, which is not defined for all
sets in $\mathscr{P}(\N)$, can be extended, but the density values in the extension
are not unique.
It is clear that the definition (\ref{eq:rhomag}) is insufficiently discriminating
and cannot be used as a general definition of a magnum. Thus, we are led to consider
alternative ways to evaluate ${\kappa_A}(\omega)$.  


\section{Some Theorems Concerning Counting Sequences}
\label{sec:theorems}

In this section, we concentrate on subsets of $\mathbb{N}$.
Larger sets will be considered in \S\ref{sec:largersets}.
In \S\ref{sec:SomeTheorems} we prove several theorems about the
counting sequences of sets in $\mathscr{P}(\mathbb{N})$.
In \S\ref{sec:isobar} we introduce the concept of
isobaric equivalence of sets.

We start with some definitions. 
\label{pg:chikap}
For a subset $A\subseteq\mathbb{N}$, the indicator sequence --- or characteristic
sequence --- $\chi_A = (\chi_A(n))_n$ is defined as
$$
\chi_A(n) = | A\cap\{n\}| 
          = \begin{cases}
            1 & \text{\ if\ } n\in A \\
            0 & \text{\ if\ } n\not\in A 
            \end{cases}
$$
For any indicator sequence $\chi = (\chi(n))_n$, we define the counting sequence
$\kappa:\mathbb{N}\to\mathbb{N}_0$ by 
$$
\kappa(n) = | A\cap I_n| = \sum_{k=1}^n \chi(k) .
$$
where $I_n = \{1,2,\dots,n\}$.  It follows that $\kappa(1) = \chi(1)$
and $\kappa(n) = \kappa(n-1) + \chi(n)$ for $n > 1$.
Note that any counting sequence $(\kappa(n))_n$ is nondecreasing,
with $\kappa(1)\in\{0,1\}$ and $\kappa(n+1)\in\{\kappa(n),\kappa(n)+1\}$.

Two counting sequences are equal, $\kappa_A = \kappa_B$, 
if and only if all corresponding terms are equal. In that case, $A=B$.
We define one counting sequence to be strictly less than another,
$\kappa_B < \kappa_A$, if and only if each term of the one
is strictly less than the corresponding term of the other.

\noindent
\textit{Sum and Product of Sequences.}

Suppose $A$ and $B$ are subsets of $\mathbb{N}$, with counting sequences
$\kappa_A = (\kappa_A(n))_n$ and $\kappa_B = (\kappa_B(n))_n$.
We define the sum and product of two sequences componentwise:
\begin{eqnarray*}
\kappa_A + \kappa_B     &:=& (\kappa_A(n) + \kappa_B(n))_n  \\
\kappa_A \cdot \kappa_B &:=& (\kappa_A(n) \cdot \kappa_B(n))_n \,.
\end{eqnarray*}

\subsubsection*{Eventual Equality and the Fr\'{e}chet Filter}

Two counting sequences may differ for a finite number of terms
but be equal for all terms beyond some point. We define
\emph{eventual equality} (denoted by $\Eeq$) of two sequences
as follows:
$$
\kappa_A \Eeq \kappa_B \iff
\exists N : \kappa_A(n) = \kappa_B(n) \text{\ for all\ } n \ge N \,.
$$
It is clear that eventual equality is an equivalence relation.
We define a sequence $\kappa_B$ to \emph{eventually less than}
$\kappa_A$ as follows:
$$
\kappa_B \Elt \kappa_A \iff
\exists N : \kappa_B(n) < \kappa_A(n) \text{\ for all\ } n \ge N \,.
$$
It is clear that this is a strict partial ordering of counting sequences.

Let $A = \{a_1, a_2, \dots, a_M\}$ be a finite subset of $\mathbb{N}$
with cardinality $\fcard{A} = M$.
The counting sequence of $A$ is
\begin{equation}
\kappa_A(n) = \fcard{A\cap I_n} \,.
\label{eq:finhash}
\end{equation}
Clearly, $\kappa_A(n)$ is eventually equal to $M$.

The discussion of eventual equality and eventual dominance is facilitated
by introducing the \emph{Fr\'{e}chet Filter}, as in \cite{Trl24}.
%
%
The Fr\'{e}chet Filter is the set of all cofinite subsets of $\mathbb{N}$:
$$
\mathcal{F} := \{A\subseteq\mathbb{N} : \mathbb{N}\setminus A\ \text{is finite}\}
$$ 
If $A$ and $B$ are subsets of $\mathbb{N}$ with counting sequences
$\kappa_A$ and $\kappa_B$, then $\kappa_A \Eeq \kappa_B$ if and
only if $A \equiv B \textrm{\ (mod\ }\mathcal{F})$, that is, $A$ and $B$
are in the same class $[\, A\, ]$ when the power set $\mathscr{P}(\mathbb{N})$
is partitioned into equivalence classes of the quotient set $\mathscr{P}(\mathbb{N})/\mathcal{F}$.
Sets in the same equivalence class are eventually equal:
$$
[\, A\, ] = \{B\subseteq\mathbb{N} : \kappa(B) \Eeq \kappa(A) \} \,.
$$
We will show later (\S\ref{sec:M-theorems}, Theorem~\ref{th:EventualEquality})
that $\kappa_A \Eeq \kappa_B$ implies $m(A) = m(B)$.

\label{sec:M-theorems}

\subsection{Some Theorems}
\label{sec:SomeTheorems}

We first show that the counting sequence of a proper subset is eventually
less than that of the set itself. This is a crucial step in establishing 
the Euclidean Principle.
\begin{theorem} 
\label{th:EP1}
Let $A$ and $B$ be subsets of $\mathbb{N}$, with counting sequences
$\kappa_A$ and $\kappa_B$. If $B\subset A$ then
$\kappa_B(n) \le \kappa_A(n)$ for all $n$.
Moreover, $\kappa_B \Elt \kappa_A$.
\end{theorem}
\begin{proof}
$B\subset A$ implies $\chi_B(n) \le \chi_A(n)$ for all $n$ and
$\chi_B(N) < \chi_A(N)$ for some $N$.
Thus, $\kappa_B(n) \le \kappa_A(n)$ for all $n$,
so $\kappa_B \le \kappa_A$.
Moreover, $\kappa_B(n) < \kappa_A(n)$ for all $n \ge N$,
so $\kappa_B \Elt \kappa_A$.
\end{proof}
The extension of this result (Theorem~\ref{th:Theorem-M1} in \S\ref{sec:M-theorems})
will guarantee the validity of the Euclidean Principle.
\begin{theorem} 
\label{th:FinAddKappa}
Let $A$ and $B$ be subsets of $\mathbb{N}$, with counting sequences
$\kappa_A$ and $\kappa_B$. Then $\kappa_{A\cup B}$, the counting sequence of
$A\cup B$, is
$$
\kappa_{A\cup B} = \kappa_A + \kappa_B - \kappa_{A\cap B}.
$$
\end{theorem}
\begin{proof}
$\chi_{A\cup B}(n) = \chi_A(n) + \chi_B(n) - \chi_{A\cap B}(n)$.
Therefore, for all $n$,
\begin{eqnarray*}
\kappa_{A\cup B}(n) &=& \sum_{k=1}^n [\chi_A(k) + \chi_B(k) - \chi_{A\cap B}(k)] \\
                    &=&     [\kappa_A(n) + \kappa_B(n) - \kappa_{A\cap B}(n)] ,
\end{eqnarray*}
which implies the required result, $\kappa_{A\cup B} = \kappa_A + \kappa_B - \kappa_{A\cap B}$.
\end{proof}
Finite additivity for the counting sequences of disjoint sets follows immediately:
\begin{corollary} 
For disjoint subsets $A$, $B$ of $\mathbb{N}$, with counting sequences
$\kappa_A$ and $\kappa_B$,
\begin{equation}
\kappa_{A\uplus B} = \kappa_A + \kappa_B \,.
\label{eq:finado}
\end{equation}
\label{cor:k1}
\end{corollary}
By an inductive argument, this yields a more general result:
\begin{corollary} 
For a finite set of disjoint subsets $\{A_j : j=1, 2, \dots, N \}$
of $\N$  with counting sequences $\kappa(A_j) = \kappa_{A_j}$,
$$
\kappa\biggl(\biguplus_{j=1}^N A_j\biggr) = \sum_{j=1}^N \kappa(A_j) \,.
$$
\label{cor:k2}
\end{corollary}
Theorem~\ref{th:FinAddKappa} provides us with an alternative proof
of Theorem~\ref{th:EP1}
\begin{corollary} 
Let $A$ and $B$ be subsets of $\mathbb{N}$, with counting sequences
$\kappa_A$ and $\kappa_B$. If $B\subset A$ then
$\kappa_B(n) \le \kappa_A(n)$ for all $n$.
Moreover, $\kappa_B \Elt \kappa_A$.
\end{corollary}
\begin{proof}
It is clear that $\kappa_{A\setminus B}>0$. 
Since $A = B \uplus (A\setminus B)$, it follows from Corollory~\ref{cor:k1} 
that $\kappa_A = \kappa_B + \kappa_{A\setminus B} > \kappa_B$.
\end{proof}

For the remainder of this section, we write $\kappa(A)$ for $\kappa_A$, etc.
\begin{theorem}
\label{th:lotsofstuff}
(a) Let $A\comp = \mathbb{N}\setminus A$. Then $\kappa(A\comp) = \kappa(\mathbb{N}) - \kappa(A)$. \\
(b) If $a\in A$ then $\kappa(A\setminus\{a\}) = \kappa(A) - 1$. \\
(c) If $b\notin A$ then $\kappa(A\uplus\{b\}) = \kappa(A) + 1$. \\
(d) If $a\in A$ and $b\notin A$ then $\kappa((A\setminus\{a\})\uplus\{b\}) = \kappa(A)$. \\
(e) For $n$ distinct elements $\{a_1, a_2, \dots , a_n\}$ of $A$ we have
$\kappa(A\setminus \biguplus_{k=1}^n \{a_k\}) = \kappa(A) - n$. \\
(f) For $n$ distinct elements $\{b_1, b_2, \dots , b_n\}$ not in $A$ we have
$\kappa(A\uplus \biguplus_{k=1}^n \{b_k\}) = \kappa(A) + n$. \\
(g) If $B\subset A$ then $\kappa(A\setminus B) = \kappa(A) - \kappa(B)$. \\
(h) For any $B$, $\kappa(A\setminus B) = \kappa(A) - \kappa(A\cap B)$.
\label{th:pm1}
\end{theorem}

\begin{proof}
(a) Since 
$\mathbb{N} = A \uplus A\comp$, we have  
$\kappa(\mathbb{N}) = \kappa(A) + \kappa(A\comp)$. The result follows immediately.

(b) $A = (A\setminus\{a\}) \uplus \{a\}$. Therefore,
$\kappa(A) = \kappa(A\setminus\{a\}) + \kappa(\{a\}) = \kappa(A\setminus\{a\}) + 1$.

(c) If $b\notin A$ then $A\cap\{b\}=\emptyset$ and
$\kappa(A\uplus\{b\}) = \kappa(A) + \kappa(\{b\}) = \kappa(A) + 1$.

(d) This follows immediately from (b) and (c).

(e) This follows inductively from $\kappa(\{a_1, a_2, \dots , a_n\}) = n$ and (b).

(f) This follows inductively from $\kappa(\{b_1, b_2, \dots , b_n\}) = n$ and (c).

(g) $A = (A\setminus B) \uplus B$.
So, $\kappa(A) = \kappa(A\setminus B) + \kappa(B)$ or $\kappa(A\setminus B) = \kappa(A) - \kappa(B)$.

(h) $A = (A\setminus B) \uplus (A \cap B)$. 
So, $\kappa(A) = \kappa(A\setminus B) + \kappa(A \cap B)$ or $\kappa(A\setminus B) = \kappa(A) - \kappa(A\cap B)$.
\end{proof}


\subsection{Isobaric Equivalence}
\label{sec:isobar}

It is useful to consider partitions of infinite sets that are
comprised of finite components. In \cite{BeNa19}, infinite sets
are expressed as unions of `labelled sets'.  In \cite{Trl24},
the relevant partitions are called canonical arrangements.
In this section, we partition sets into unions of disjoint finite 
components all of which are equal in size. A specific class of
partitions, called fenestrations, is considered in \S\ref{sec:GIT}.

\begin{definition}
\label{def:regfen}
A \emph{simple partition} $\mathscr{W} = \{W_k:k\in\N\}$ 
is a collection of disjoint finite sets or `windows',
all of equal length $L$, that partition $\N$, such that
\begin{equation}
\mathbb{N} = \biguplus_{k\in\mathbb{N}} W_k \qquad\mbox{where}\qquad
W_k = \{ (k-1)L +1, \dots ,k L \} \,.
\label{eq:fenes}
\end{equation}
\end{definition}
The partition $\mathscr{W}$ of $\N$ induces a partition of any set
$A\in\mathscr{P}(\mathbb{N})$ into a countable union of disjoint,
finite subsets, $A_k = A\cap W_k$, with $A = \biguplus_k A_k$. 
We define a sequence of weights
\begin{equation}
\mathit{w}_A = \bigl( \fcard{A_1}, \fcard{A_2}, \dots , \fcard{A_k}, \dots \bigr) \,,
\label{eq:Aweights}
\end{equation}
whose $k$-th term gives the number of elements of $A$ in window $W_k$.

We now introduce an equivalence relation --- isobaric equivalence --- between sets having
equal weight sequences.  We denote isobary of two sets under a simple partition $\mathscr{W}$
by writing $A_1 \isobar A_2 (\text{mod\ }L)$, or simply $A_1 \isobar A_2$.%
\footnote{The term \emph{isobary} is from 
$\iota\sigma o\ \beta\alpha\rho o\varsigma$, `equal weight'.}

\begin{definition}
Two sets $A_1$, $A_2$ in $\mathscr{P}(\N)$ are \emph{isobaric} if,
for some simple partition $\mathscr{W}$, with window length $L\in\N$,
they have equal weight sequences; that is, if $\mathit{w}_{A_1} = \mathit{w}_{A_2}$.
\end{definition}
\begin{theorem}
Isobary is an equivalence relation.
\end{theorem}
\begin{proof}
Isobary is obviously reflexive ($A_1\isobar A_1$) and symmetric
($A_1\isobar A_2\iff A_2\isobar A_1$).  It is also transitive:
if $A_1\isobar A_2$ with window length $L_1$ and $A_2\isobar A_3$
with window length $L_2$, then $A_1\isobar A_3$ with window length
$L_1 L_2$.  Thus, $\isobar$ is an equivalence relation.
\end{proof}
We will show later (Theorem~\ref{th:L-iso} in \S\ref{sec:GIT}) that
any two isobaric sets have equal magnums.

\newcommand{\vequal}{\rotatebox[origin=c]{90}{=}}
\renewcommand{\hat}{\widehat}
\renewcommand{\tilde}{\widetilde}




\section{Extension of the Counting Function}
\label{sec:extension}

We aim to define the magnum of a set $A\subseteq\N$ as the value of the counting
function of $A$ for argument $\omega$, so we must extend the domain of
$\kappa_A:\mathbb{N}\to\mathbb{N}_0$.  We now show how this can be done.
For any function, a structure is prescribed, that is, an algorithm,
recipe or rule by which a value is determined from a given input
argument.  A domain of definition must also be specified.
We consider extensions in which \emph{the structure, or functional form,
remains unchanged while the domain is enlarged.}

We assume that the functional form of $f:\mathbb{N}\to\mathbb{N}_0$
is known so that, for any $n\in\mathbb{N}$, we can evaluate $f(n)$.
Then we can extend $f$ to $\Nn$, so that, in particular, it may be evaluated
for argument $\omega$. 
In this manner, the counting function $\kappa_A$ can be extended to a nondecreasing
function on the surnatural numbers, $\hat{\kappa_A} : \Nn\to\Nn$ and we can
evaluate it to obtain the magnum of a set $A$ as $m(A) = \hat{\kappa_A}(\omega)$.

As discussed in the Introduction, there are difficulties in assuming a
Transfer Principle, which posits that properties that hold for real numbers
remain valid when \emph{transferred} to the surreal domain. 
The theory presented below depends only upon a much weaker assumption,
the Axiom of Extension which, in many cases of interest, is self-evident.

\subsection{The Extension of the Domain of Definition of Functions}

In many cases, a sequence $a:\mathbb{N}\to\mathbb{N}_0$
can be extended in an obvious and natural way to a sequence on the surnatural numbers,
$\hat a:\Nn\to\Nn$. The simplest example is a sequence $\mathrm{c}_k(n) = k$,
which is constant for all $n\in\mathbb{N}$; this extends to the constant function
$\hat{\mathrm{c}_k}(\nu)=k$ for all $\nu\in\Nn$.
The identity function $\mathrm{id}_\mathbb{N}(n)=n$
extends to $\mathrm{id}_{\Nn}(\nu)=\nu$ for all $\nu\in\Nn$.
Equally obvious are the extensions of polynomial functions, for which
$a(n)=\sum_{j=1}^k c_j n^j$ on $\mathbb{N}$ becomes
$\hat a(\nu)=\sum_{j=1}^k c_j \nu^j$ on $\Nn$.
The inverse of a strictly increasing polynomial also extends canonically so that,
for example, if $f(n)=k n^2$ with inverse $f^{-1}(n) = \sqrt{n/k}$, the inverse of the
extension is $\hat{f}^{-1}(\nu) = \sqrt{\nu/k}$.%

It must be noted that the extension of a function to a larger domain may not be unique.
For example, the identity function and the modulus or absolute value function
are identical on $\N$ but differ on $\Z$. To avoid ambiguity, it is essential
to specify precisely the \emph{form of the function} --- or the algorithm --- 
for which the domain is to be extended.

The counting sequence of a finite set $F$ with cardinality $M$ is
eventually constant.  The {extension} of $\kappa_{F} : \mathbb{N}\to\mathbb{N}_0$
is a function $\hat{\kappa_{F}} : \Nn\to\Nn$ such that
$$
\hat{\kappa_{F}}(\nu) = 
\begin{cases}
\kappa_{F}(\nu) & \text{for all\ }\nu \in\mathbb{N} \,, \\
    M       & \text{for all\ }\nu \in \Nn\setminus\mathbb{N} \,.
\end{cases} 
$$


\subsubsection*{The Axiom of Extension}

We denote the extension of a function $f:\N\to\N_0$ by $\widehat f:\Nn\to\Nn$.
To ensure that key properties of functions on the domain $\N$ extend to the
domain $\Nn$ of surnatural numbers, we need to posit an axiom:

\textsc{The Axiom of Extension:}
\label{ax:exten}
{\em 
For any nondecreasing functions $f:\N\to\N_0$ and $g:\N\to\N_0$, there exist
extensions $\hat f:\Nn\to\Nn$ and $\hat g:\Nn\to\Nn$ such that
\begin{eqnarray*}
f(n) \Eeq g(n) &\implies& \hat f(\nu) = \hat g(\nu) \text{\ for\ } \nu\in\Nn\setminus\N \\
f(n) \Elt g(n) &\implies& \hat f(\nu) < \hat g(\nu) \text{\ for\ } \nu\in\Nn\setminus\N 
\end{eqnarray*}
and the extensions preserve sums, products and (where defined) compositions:
\begin{equation}
\widehat{(f+g)} = \widehat{f} + \widehat{g} ,
\qquad
\widehat{(f\cdot g)} = \widehat{f} \cdot \widehat{g} 
\qquad\text{and}\qquad
\widehat{(f\circ g)} = \widehat{f} \circ \widehat{g} .
\nonumber
\end{equation}
}  

The extension of some transcendental functions from the real to the surreal domain
is discussed in \cite{ONAG} and in \cite{RSS14}.  In the applications below,
we envoke the Axiom of Extension only for nondecreasing functions.
If attention is confined to polynomial, logarithmic and exponential functions,
the conditions above are automatically satisfied and \emph{no axiom is required.}

\subsubsection*{There are no Limits in $\Nn$}

The extension of a function $f:\mathbb{N}\to\mathbb{N}$
to $\hat f:\Nn\to\Nn$ can lead to surprising consequences.
As Conway observed \cite[pg.~42]{ONAG}, ``limits don't seem to work \dots [in $\No$]''.
Grandi's series is $ G = 1 - 1 + 1 - 1 + \cdots $, which does not converge.
The sequence of partial sums is $ S(n) = (1,0,1,0, \dots) = (1 - (-1)^n)/{2}$.
This extends in an obvious way to the domain of surnatural numbers,
$\hat S(\nu) = (1 - (-1)^\nu)/{2}$ and, since $\omega$ is
an even number, $\hat S(\omega) = 0$. Thus, while $\lim_{n\to\infty} S(n)$
does not exist, the extension of $S$ has a definite value at $\omega$.

\subsection{The Extension of Subsets of $\mathbb{N}$ to Subsets of $\Nn$}

For any function $f:\mathbb{N}\to\mathbb{N}_0$, 
the Axiom Schema of Separation \cite[p.~7]{Jech02} allows us to express the
image set $f(\mathbb{N})$ as $A = \{m\in\mathbb{N}_0 : \exists n\ (n\in\mathbb{N})\land(m=f(n)) \}$,
while the Axiom Schema of Replacement \cite[p.~13]{Jech02} allows us to
express the set as $A = \{ f(n):n\in\mathbb{N}\}$. Under the ZF axioms, these sets
are identical; we may alternate between the two formulations as convenient.
%
%
\begin{definition}
If $f$ extends to $\hat f:\Nn\to\Nn$, the \emph{extension} of the set
$A = \{ f(n):n\in\mathbb{N}\}$ is $\hat A = \{\hat f(\nu):\nu\in\Nn\}$.
Alternatively, the same set expressed as
$A = \{m\in\mathbb{N}_0 : \exists n\ (n\in\mathbb{N})\land(m=f(n))\}$ extends to
$\hat A = \{\mu\in\Nn : \exists\nu\ (\nu\in\Nn) \land (\mu=\hat f(\nu))\}$. 
\label{def:Ahat}
\end{definition}
Clearly, $\hat{\mathbb{N}} = \Nn$.
Other simple examples are 
$\hat{2\mathbb{N}} = 2\Nn$ and $\hat{2\mathbb{N}-1} = 2\Nn-1$.%
\footnote{The concept of parity extends naturally from
$\N$ to $\Nn$. In particular, the surnatural number $\omega$ is even.}
\begin{theorem}
\label{th:AhatBhat}
For $A$ and $B$ in $\mathscr{P}(\mathbb{N})$,
with extensions $\hat A$ and $\hat B$ to $\Nn$, we have:
\begin{itemize}
\item[(a)] $\hat{A\cap B} = \hat A \cap \hat B$.
\item[(b)] $\hat{A\cup B} = \hat A \cup \hat B$.
\item[(c)] ${A\subseteq B} \iff \hat A \subseteq \hat B$.
\item[(d)] ${\hat{(A\comp)}} = (\hat A)\Comp = \Nn\setminus \hat A$.
\item[(e)] ${\hat{(A\setminus B)}} = \hat A\setminus\hat B$.
\end{itemize}
\label{th:hatcap}
\end{theorem}
\begin{proof}
(a)
The proof uses the set-theoretic relationship
$\{\mu:P_1(\mu)\} \cap \{\mu:P_2(\mu)\} = \{\mu:P_1(\mu)\land P_2(\mu)\}$.
We write 
$A = \{m\in\mathbb{N} :\exists n\in\N\ (m=a(n))\}$ and
$B = \{m\in\mathbb{N} :\exists n\in\N\ (m=b(n))\}$.
Thus,
$\hat A = \{\mu\in\Nn:\exists \nu\in\Nn\ (\mu=\hat a(\nu))\}$ and
$\hat B = \{\mu\in\Nn:\exists \nu\in\Nn\ (\mu=\hat b(\nu))\}$.
Then
\begin{eqnarray*}
\hat{A\cap B}
&=& 
\{\mu\in\Nn:[\exists\nu_1\in\Nn\ (\mu=\hat a(\nu_1))]\land
            [\exists\nu_2\in\Nn\ (\mu=\hat b(\nu_2))]\} \\
&=& 
\{\mu\in\Nn:\exists \nu_1\in\Nn\ (\mu=\hat a(\nu_1)) \} \cap
\{\mu\in\Nn:\exists \nu_2\in\Nn\ (\mu=\hat b(\nu_2)) \} \\
&=& 
\hat A \cap \hat B \,.
\end{eqnarray*}
The remaining parts of the theorem are proved in a similar manner.
\end{proof}
Theorem~\ref{th:AhatBhat} corresponds to Proposition 2.15 in \cite[p.~21]{BeNa19},
where proofs in the context of the hyperreal numbers are given.


\subsubsection*{Omega Sets}

\begin{definition}
A set $A\subseteq\mathbb{N}$ is an \emph{omega set} 
if and only if $\omega\in\hat A$.
The family of all sets in $\mathscr{P}(\mathbb{N})$ that are omega sets is denoted by $\BOM$.
\label{def:omset}
\end{definition}

From this definition and Theorem~\ref{th:AhatBhat} it follows that if $A\subseteq\mathbb{N}$
is an omega set, then $A\comp$ is not. For example, $2\mathbb{N} = \{ 2n:n\in\mathbb{N}\}$,
so $\hat{2\mathbb{N}} = \{ 2\nu:\nu\in\Nn\}$ and, since $\nu=\omega/2\in\Nn$,
$\omega\in\hat{2\mathbb{N}}$.
However, $2\mathbb{N}-1 = \{ 2n-1:n\in\mathbb{N}\}$,
so $\hat{2\mathbb{N}-1} = \{ 2\nu-1:n\in\Nn\}$,
but $2\nu-1=\omega \implies \nu = \frac{\omega}{2}+\frac{1}{2} \not\in\Nn$ so
$\omega\not\in\hat{2\mathbb{N}-1}$.  Thus, the set of even natural
numbers is an omega set, whereas the set of odd natural numbers is not.

Benci \&\ DiNasso \cite[p.~221]{BeNa19} constructed a model for their Numerosity Theory
by means of an equivalence relation on $\mathbb{R}^{*}=\mathbb{R}^\mathbb{N}/\mathcal{U}$,
where $\mathcal{U}$ is a free ultrafilter on $\mathbb{N}$.
Although we do not explore this approach below, it is of interest to observe that
the family $\BOM$ is a free ultrafilter on $\mathbb{N}$. 
To prove this, it must be shown that:
\begin{itemize}
\item[(a)] $\varnothing\not\in\BOM$ and $\mathbb{N}\in\BOM$.
\item[(b)] $(A\in\BOM)\land(B\supseteq A)\implies B\in\BOM$.
\item[(c)] $A, B \in\BOM \implies (A\cap B)\in\BOM$.
\item[(d)] If $A\in\BOM$ and $A=\biguplus_{k=1}^n A_k$,
           then $A_j\in\BOM$ for exactly one $j$.
\item[(e)] $A\in\BOM \iff A\comp\not\in\BOM$.
\item[(f)] For $A$ finite, $A\not\in\BOM$.
\end{itemize}
All of these properties follow directly from the definition of omega sets,
so we may state the Ultrafilter Theorem:
\begin{theorem}
The family $\BOM$ of all sets of natural numbers that are omega sets
is a free ultrafilter on $\mathbb{N}$.
\label{th:UltraFilter}
\end{theorem}

Omega sets as defined above are the counterparts of the \emph{qualified sets} of
Benci and DiNasso and Theorem~\ref{th:UltraFilter} corresponds to their
Proposition 2.22 \cite[pg.~24]{BeNa19}.
However, in the context of the surnatural numbers,
$\omega$ is \emph{even} since $\frac{\omega}{2}\in\Nn$. There is no automatic
guarantee that the hyperreal number $\boldsymbol{\alpha}$ of Benci \emph{et al.}\
is even, so an additional assumption, the Qualified Set Axiom, is required
\cite[p.~36]{BeNa19}.


\section{Definition of the Magnum}
\label{sec:DefMag}


We are now in a position to define the magnum of a set $A$ of natural numbers:
\begin{definition}
\label{def:magnum}
Let $A$ be a set of natural numbers, $\kappa_A(n)$ the counting function of $A$ and
$\hat{\kappa_A}(\nu)$ the extension of this function to $\Nn$.  Then the magnum of $A$
is defined by
\begin{equation}
\boxed{
m(A): = \hat{\kappa_A}(\omega) \,.
\label{eq:MagDef}
}  
\end{equation}
\end{definition}

The Axiom of Extension implies that the counting function of any set
$A\in\mathscr{P}(\N)$
can be extended, so that its magnum is given by (\ref{eq:MagDef}).

\subsection{Some Theorems for Magnums of Sets in $\mathscr{P}(\N)$}
\label{sec:M-theorems}

In \S\ref{sec:SomeTheorems} we stated and proved several theorems for 
counting sequences of sets in $\mathscr{P}(\mathbb{N})$.
We now consider the corresponding results for set magnums.
We begin by showing that changing a finite number of terms in the counting
sequence of a set has no effect on the magnum of the set.
\begin{theorem}
Let $A$ and $B$ be two subsets of $\mathbb{N}$ such that
$\kappa_A \Eeq \kappa_B$.  Then $m(A) = m(B)$.
\label{th:EventualEquality}
\end{theorem}
\begin{proof}
By hypothesis, there exists $N\in\N$ such that $\kappa_A(n) = \kappa_B(n)$
for $n\ge N$. Appealing to the Axiom of Extension, it follows that
$\hat{\kappa_A}(\nu) = \hat{\kappa_B}(\nu)$ for $\nu\in\Nn\setminus\N$, so that
$\hat{\kappa_A}(\omega) = \hat{\kappa_B}(\omega)$ and therefore $m(A) = m(B)$.
\end{proof}
It follows immediately that finite sets having the same cardinality have equal magnums.
\begin{corollary}
Let $A$ and $B$ be two finite subsets of $\mathbb{N}$,
both having the same cardinality $M$. Then $m(A) = m(B)$.
\end{corollary}
\begin{proof}
The counting sequences of the two sets are 
$$
\kappa_A(n) = (\fcard{A\cap I_n}) \qquad\text{and}\qquad \kappa_B(n) = (\fcard{B\cap I_n}) 
$$
where $I_n = \{1,2,\dots,n\}$.
Since $\fcard{A}=\fcard{B}=M$, there exists $N\in\N$ such that $\kappa_A(n) = \kappa_B(n)$
for all $n\ge N$.  Appealing to the Axiom of Extension, it follows that
$\hat{\kappa_A}(\omega) = \hat{\kappa_B}(\omega)$ so that, using (\ref{eq:MagDef}),
$m(A) = m(B) = M$.
\end{proof}


\begin{theorem}
Let $A$ and $B$ be two subsets of $\mathbb{N}$ with counting sequences
$\kappa_A(n)$ and $\kappa_B(n)$, such that
$\forall n\ge N, \kappa_A(n) \ne \kappa_B(n)$. 
Then $m(A) \ne m(B)$.
\label{th:EventualInequality}
\end{theorem}
\begin{proof}
%
%
First let us assume that $\kappa_A(n) \ne \kappa_B(n)$ for all $n$.
Since $\kappa_A(n)$ and $\kappa_B(n)$ increase in unit steps, the difference
$\kappa_A(n)-\kappa_B(n)$ can only change by a unit for each increment of $n$,
so that either $\kappa_A(n)<\kappa_B(n)$ for all $n$ or $\kappa_A(n)>\kappa_B(n)$
for all $n$.  Consequently, 
$\forall \nu\in\Nn, (\widehat{\kappa_A}(\nu)<\widehat{\kappa_B}(\nu))
              \lor(\widehat{\kappa_A}(\nu)>\widehat{\kappa_B}(\nu))$, 
so that $m(A) \ne m(B)$.

Since, by Theorem~\ref{th:EventualEquality}, changing a finite number of terms in
a counting sequence has no effect on the magnum of the set, we conclude that if
$\kappa_A(n) \ne \kappa_B(n)$ for $n\ge N$, then $m(A) \ne m(B)$.
\end{proof}


We now prove the Euclidean Principle, ensuring that the magnum of
a proper subset $B$ of any set $A$ is less than the magnum of $A$ itself.

\begin{theorem} 
\label{th:Theorem-M1}
Let $A$ and $B$ be subsets of $\mathbb{N}$ with $B\subset A$. Then $m(B) < m(A)$.
\end{theorem}
\begin{proof}
In Theorem~\ref{th:EP1} we proved that, if the counting sequences
of $A$ and $B$ are $\kappa(A)$ and $\kappa(B)$ and $B\subset A$,
then $\kappa(B) \Elt \kappa(A)$.  By the Axiom of Extension, this implies 
$\hat{\kappa_B}(\nu) < \hat{\kappa_A}(\nu)$ for all
$\nu\in\Nn\setminus\N$ and, in particular,
$\hat{\kappa_B}(\omega) < \hat{\kappa_A}(\omega)$. Therefore, $m(B) < m(A)$.
\end{proof}

\begin{theorem} 
\label{th:Theorem-M2}
Let $A$ and $B$ be subsets of $\mathbb{N}$, with magnums
$m(A)$ and $m(B)$. Then
$$
m(A\cup B) = m(A) + m(B) - m(A\cap B).
$$
If  $A$ and $B$ are disjoint, then $m(A\uplus B) = m(A) + m(B)$.
\end{theorem}
\begin{proof}
This follows directly from Theorem~\ref{th:FinAddKappa} and the definition
of set magnums.
\end{proof}

\begin{theorem} 
\label{th:Theorem-M3}
(a) Let $A\comp = \mathbb{N}\setminus A$. Then $m(A\comp) = \omega - m(A)$. \\
(b) If $a\in A$ then $m(A\setminus\{a\}) = m(A) - 1$. \\
(c) If $b\notin A$ then $m(A\uplus\{b\}) = m(A) + 1$. \\
(d) If $a\in A$ and $b\notin A$ then $m((A\setminus\{a\})\uplus\{b\}) = m(A)$. \\
(e) For $n$ distinct elements $\{a_1, a_2, \dots , a_n\}$ of $A$ we have
$m(A\setminus \biguplus_{k=1}^n \{a_k\}) = m(A) - n$. \\
(f) For $n$ distinct elements $\{b_1, b_2, \dots , b_n\}$ not in $A$ we have
$m(A\uplus \biguplus_{k=1}^n \{b_k\}) = m(A) + n$. \\
(g) If $B\subset A$ then $m(A\setminus B) = m(A) - m(B)$. \\
(h) For any $B$, $m(A\setminus B) = m(A) - m(A\cap B)$.
\end{theorem}

\begin{proof}
These results follow immediately from Theorem~\ref{th:lotsofstuff}
and the definition of set magnums.
\end{proof}




\subsubsection*{The Density Theorem}
\label{sec:DenTheorem}

We now prove the Density Theorem, showing the connection between the
density and magnum of any set $A\subset\N$ having density $\rho(A)$.
Using the normal form [Appendix~1, (\ref{eq:normalform})] we can express
any magnum as $m(A) = c\,\omega + o(\omega)$ where $c\in[0,1]$ and
$o(\omega)/\omega$ is infinitesimal. 
%
%

\begin{lemma}\label{lemma:ratio}
  If $f$ and $g$ are non-decreasing functions from $\N$ to $\N_0$ and
  $\lim_{n\to\infty} {f(n)}/{g(n)}=L$ then for $\nu\in \Nn
  \setminus \N$, we have ${\hat
    f(\nu)}/{\hat g(\nu)}\approx L$  or equivalently $\hat f(\nu) =
  \hat g(\nu)L + o(\hat g(\nu))$. 
\end{lemma}

\begin{proof}
  The convergence of the sequence implies that, given any real number $\epsilon>0$,
  there exists $n_0$ such that $f(n)<(L+\epsilon)g(n)$ for all $n\ge
  n_0$ so by the Axiom of Extension, $\hat f(\nu)< (L+\epsilon)\hat
  g(\nu)$ and ${\hat f(\nu)}/{\hat g(\nu)} < L+\epsilon$ for all
  $\nu\in \Nn\setminus \N$.   Similarly ${\hat f(\nu)}/{\hat g(\nu)} > L-\epsilon$ for all
  $\nu\in \Nn\setminus \N$.  It follows ${\hat f(\nu)}/{\hat g(\nu)} \approx L$ for all
  $\nu\in \Nn\setminus \N$.
\end{proof}

\begin{theorem}
Let $A\subset \N$ have density $\rho(A)\in[0,1]$. Then $m(A) = \rho(A)\omega + o(\omega)$.
\label{th:DensityTheorem}
\end{theorem}

\begin{proof}
    We apply the Lemma to $f(n)=\kappa_A(n)$ and $g(n)=\kappa_{\N}(n)=n$ at
	$\nu=\omega$.  Since $\lim_n
  {f(n)}/{g(n)}=\rho(A)$ we find that  ${m(A)}/{\omega}=
	{\hat{\kappa_A}(\omega)}/{\hat{\kappa_\N}(\omega)} \approx
        \rho(A)$ or $m(A)=\rho(A)\omega + o(\omega)$.
\end{proof}

\subsection[Consistency of the Counting Function and Magnum Form]
           {Consistency of the Counting Function and Magnum Form}
\label{sec:Equiv}

We have developed magnums genetically in terms of magnum forms (\S\ref{sec:gendef}) and
also defined them in terms of the counting function (Definition~\ref{def:magnum}):
\begin{itemize}
\item \textbf{Genetic:} $m(A)$ is constructed from the form $\{m(B):B\subset A\ |\ m(C):C\supset A\}$.
\item \textbf{Functional:} $\mu(A)$ is defined from the extended counting function:
                   $\mu(A) = \hat\kappa(\omega)$.
\end{itemize}
We must now show that these two definitions are consistent: $m(A) = \mu(A)$. 

For finite sets it is easily seen that the two approaches agree.
Let  $A = I_k = \{1,2,\dots,k\}$.
Then $m(A) = \{ 0, 1, 2, \dots k-1 \ |\ \} = k$ and $\kappa_A(k) = k$,
which is extended with constant value $k$ so that $\mu(A) =\hat{\kappa_A}(\omega) = k$. 
For cofinite sets the equivalence can also be shown. Let $A = \mathbb{N}\setminus I_k$.
$$
\kappa_A(n) = \lfloor\alpha_A^{-1}(n)\rfloor = \alpha_A^{-1}(n) = 
\begin{cases} 0, & n\le k \\
            n-k, & n > k
\end{cases}
$$
which extends to $\mu(A) = \hat{\kappa_A}(\omega) = \omega-k$.
The genetic magnum form takes the same value:
$$
m(A) = \{ 0, 1, 2, \dots \ |\ \omega, \omega-1,\dots,\omega -k+1 \} = \omega-k \,.
$$
So, for all finite and cofinite sets, the two methods yield equivalent results.
We also know that for $A=\mathbb{N}$, both methods agree: $m(A) = \mu(A) = \omega$
and, for the odd and even numbers, $m(A) = \mu(A) = \frac{\omega}{2}$. 

More generally, suppose the magnum of a set $A$ is given by
$\mu(A) = \hat{\kappa_A}(\omega)$.  By definition, $\mu(A) \in \Nn$. Therefore,
$$
\mu(A) = \{ \mu(A)-1\ |\ \mu(A)+1 \} \,.
$$
For any nonempty set $A$, the largest proper subset $B$ of $A$ has magnum $\mu(A)-1$
while the smallest proper superset $C$ has magnum $\mu(A)+1$. For example, let
$a\in A$ and $d\not\in A$ and let $B = A \setminus\{a\}$ and $C = A \uplus \{d\}$.
Then we can write
$$
\mu(A) = \{ \mu(B)\ |\ \mu(C) \} = \{ \mu(A)-1\ |\ \mu(A)+1 \} \,.
$$
Clearly, including other subsets of $A$ on the left or supersets on the right
will not alter the value of the magnum of $A$.

We conclude that the two representations of the magnum of $A$,
as $\mu(A) = \hat{\kappa_A}(\omega)$ and as
$$
m(A) = \{ m(B) : B\subset A \text{ and } B\in\mathscr{M} \ |\ 
          m(C) : C\supset A \text{ and } C\in\mathscr{M}  \}
$$
are compatible with each other.




\subsection{The General Isobary Theorem}
\label{sec:GIT}


We first prove that sets that are isobaric under a \emph{simple partition}
have equal magnums.  The proof of this will lead us to a more general result.
\begin{theorem}
Any two sets $A$ and $B$ that are isobaric under a
simple partition have equal magnums, $m(A) = m(B)$.
\label{th:L-iso}
\end{theorem}
\begin{proof}
Let $A$ and $B$, with counting functions $\kappa_A(n)$ and $\kappa_B(n)$,
be isobaric under a simple partition with constant window-length $L$.
We define $\Delta(n) = \kappa_A(Ln)-\kappa_B(Ln)$.  The hypothesis
implies $\Delta(n)=0$ for all $n$.
We now assume that $\kappa_A(n)$ and $\kappa_B(n)$ are extended to the domain
$\Nn$ as $\hat{\kappa_A}(\nu)$ and $\hat{\kappa_B}(\nu)$.

Since, for functions $f$ and $g$, we have
$\widehat{(f+g)}(\nu) = \widehat{f}(\nu) + \widehat{g}(\nu)$, we can now write 
$\hat \Delta(\nu) = \hat{\kappa_A}(L\nu) - \hat{\kappa_B}(L\nu)$.
Since $\Delta(n)$ vanishes for all $n$, its extension vanishes for all $\nu\in\Nn$.
This implies $\hat{\kappa_A}(L\nu) = \hat{\kappa_B}(L\nu)$.
In particular, setting $\nu = \omega/L$ we have
$\hat{\kappa_A}(\omega) = \hat{\kappa_B}(\omega)$.
Therefore, $m(A) = m(B)$.
\end{proof}
While this theorem does not yield direct information about the actual values
of $m(A)$ and $m(B)$, it enables us to compare the sizes of sets. For
example, the magnums of $2\mathbb{N}$ and $2\mathbb{N}-1$ are equal,
since these sets are isobaric. Consequently, since 
$2\mathbb{N}\uplus(2\mathbb{N}-1) = \mathbb{N}$, it follows that
$m(2\mathbb{N}) = m(2\mathbb{N}-1) = \omega/2$, as already shown above.

\subsubsection*{A More Powerful Result: the General Isobary Theorem}

Examining the proof of Theorem~\ref{th:L-iso}, we can see that a much more
sweeping result holds.  We require a decomposition of sets that is more general
than simple partition with components of equal length.
We will define a \emph{fenestration} of $\N$ by specifying
a \emph{partition function} $\Lambda:\N\to\N$.
Let $(\Lambda(n))_n$ be an arbitrary, strictly increasing sequence of natural numbers.
We write $\Lambda(0) = 0$ and define
$$
W_n = \{\Lambda(n-1)+1, \dots , \Lambda(n) \} = (\Lambda(n-1),\Lambda(n)]_\mathbb{N} \,.
$$
This determines a partition of the natural numbers, $\mathbb{N} = \biguplus_n W_n$.
The \emph{length} of $W_n$ is $L_n = \fcard{W_n} = \Lambda(n)-\Lambda(n-1)$.
The weight sequence of a set $A$ is $w_A = (\fcard{A \cap W_n})_n$.
It is convenient to use the symbol $\Lambda$ for both the partition function
and its range $\Lambda(\N)$, which comprises the \emph{endpoints} of the windows.

Let $\widehat\Lambda(\nu)$ be the extension of $\Lambda:\mathbb{N}\to\mathbb{N}$
to $\widehat\Lambda:\Nn\to\Nn$.  Observe that, for a simple partition, $\Lambda(n)=nL$
extends to $\widehat\Lambda(\nu)=\nu L$, so that $\omega = \widehat\Lambda(\nu)$ where
$\nu = \omega/L \in\Nn$ and $\Lambda\in\BOM$.  
To prove the general isobary theorem, we must ensure that $\omega$ is an
image point of the extension; that is, $\Lambda=\{\Lambda(n):n\in\N\}$ is an omega set
so that $\omega\in\hat\Lambda(\Nn)$.

To illustrate the need for $\Lambda$ to be an omega set,
we consider the sets $2\mathbb{N}$ and $2\mathbb{N}+1$.
Under the partition determined by $\Lambda(n) = 2n-1$, these sets are isobaric.
However, we expect that 
$m(2\mathbb{N}) = \frac{\omega}{2}$ while $m(2\mathbb{N}+1) = \frac{\omega}{2}-1$.
Indeed, the extension is $\hat{\Lambda}(\nu) = 2\nu-1$ so
$\hat{\Lambda}(\nu) = \omega$ implies $\nu = \frac{\omega}{2}+\half \not\in\Nn$,
so that $\omega\not\in\hat{\Lambda}(\Nn)$: the partition is not a fenestration.

\begin{definition}
\label{pg:deffen}
A \emph{fenestration} of $\N$ is a partition $\mathscr{W}$ of $\mathbb{N}$,
whose partition function $(\Lambda(n))_n$ is strictly increasing,
and such that its endpoint set $\Lambda$ is an omega set;
that is, $\Lambda\in\BOM$ and $\omega\in\hat{\Lambda}$.
\label{def:genfen}
\end{definition}
Clearly, simple partitions (see Def.~\ref{def:regfen} in \S\ref{sec:isobar}) are included in
the family of fenestrations since, for all $L\in\mathbb{N}$ we have $\omega/L\in\Nn$.
Isobary is now defined with respect to a fenestration:
\begin{definition}
Two sets $A$ and $B$ in $\mathscr{P}(\N)$ are \emph{isobaric} with respect to
a fenestration $\mathscr{W} = \{W_k:k\in\N\}$ if
they have equal weight sequences, that is, if
$w_A(n) = w_B(n)$ for all $n\in\mathbb{N}$.
\end{definition}
Of course, equal weight sequences imply that $\kappa_A(\Lambda(n)) = \kappa_B(\Lambda(n))$.


We now state and prove the General Isobary Theorem:
\begin{theorem}
Let $\mathscr{W}$ be a fenestration of $\N$ with endpoint set $\Lambda(\N)$.
Let $A$ and $B$ be two sets that are isobaric under $\mathscr{W}$.
Then $A$ and $B$ have equal magnums, $m(A) = m(B)$.
\label{th:iso}
\end{theorem}
\begin{proof}
Let two sets $A$ and $B$, with counting functions $\kappa_A(n)$ and $\kappa_B(n)$,
be isobaric under the fenestration determined by $\Lambda(n)$.
We define $\Delta(n) = \kappa_A(\Lambda(n))-\kappa_B(\Lambda(n))$.
The hypothesis    
implies $\Delta(n)=0$ for all $n$.
We assume that $\kappa_A(n)$ and $\kappa_B(n)$ are extended to the domain
$\Nn$ as $\hat{\kappa_A}(\nu)$ and $\hat{\kappa_B}(\nu)$.

Using the Axiom of Extension, we can write 
$\hat \Delta(\nu) = \hat{\kappa_A}(\Lambda(\nu)) - \hat{\kappa_B}(\Lambda(\nu))$.
Since $\Delta(n)$ vanishes for all $n$, its extension vanishes for all $\nu\in\Nn$.
This implies $\hat{\kappa_A}(\hat{\Lambda}(\nu)) = \hat{\kappa_B}(\hat{\Lambda}(\nu))$.
In particular, setting $\nu = \Lambdahatinv(\omega)\in\Nn$, we have
$\hat{\kappa_A}(\omega) = \hat{\kappa_B}(\omega)$ or $m(A) = m(B)$.
\end{proof}

To prove that general isobary is an equivalence relation,
we first show that the intersection of the endpoint sets
of two fenestrations of $\N$ is the endpoint set of a
fenestration of $\N$.
\begin{theorem}
Let $\mathscr{W}_1$ and $\mathscr{W}_2$ be two fenestrations of $\N$,
with respective endpoint sets $\Lambda_1$ and $\Lambda_2$.
Then the partition of $\N$ whose endpoints comprise the intersection
$\Lambda = \Lambda_1\cap\Lambda_2$ is a fenestration of $\mathbb{N}$.
\label{th:LamInt}
\end{theorem}
\begin{proof}
We define $\Lambda = \Lambda_1\cap\Lambda_2
= \{\Lambda_1(n):n\in\mathbb{N}\}\cap\{\Lambda_2(n):n\in\mathbb{N}\}$.
Since $\omega\in\hat{\Lambda_1}(\Nn)$ and $\omega\in\hat{\Lambda_2}(\Nn)$
we have $\omega\in\hat{\Lambda_1}\cap\hat{\Lambda_2}$.
By Theorem~\ref{th:AhatBhat}(a),
$\hat{\Lambda_1}\cap\hat{\Lambda_2} = \hat{\Lambda_1\cap\Lambda_2} = \hat\Lambda$.
Therefore, $\omega\in\hat{\Lambda}(\Nn)$, so that $\Lambda$ is an omega set.
Any finite set $A$ has $\hat A = A$, so $\Lambda$ must be an infinite set.
Thus, $\Lambda = \Lambda_1\cap\Lambda_2$ is the endpoint set of a
fenestration of $\N$.
\end{proof}
We now proof that general isobary is an equivalence relation.
\begin{theorem}
Let $\mathscr{W}_1$ and $\mathscr{W}_2$ be two fenestrations, with respective
endpoint sets $\Lambda_1$ and $\Lambda_2$.
For $A$, $B$ and $C$ in $\mathscr{P}(\mathbb{N})$, assume that
$A$ and $B$ are isobaric under $\mathscr{W}_1$ and 
$B$ and $C$ are isobaric under $\mathscr{W}_2$.
Then $A$ and $C$ are isobaric under the fenestration $\mathscr{W}$
having endpoint set $\Lambda = \Lambda_1\cap\Lambda_2$.
It follows that isobary is an equivalence relation.
\label{th:isotrans}
\end{theorem}
\begin{proof}
Let sets $A$ and $B$ be isobaric under a fenestration $\mathscr{W}_1$
with endpoint set $\Lambda_1$ and let $B$ and $C$ be isobaric under
a fenestration $\mathscr{W}_2$ with endpoint set $\Lambda_2$.
By Theorem~\ref{th:iso}, $A$ and $B$ have equal magnums, $m(A) = m(B)$
and $B$ and $C$ have equal magnums, $m(B) = m(C)$. Therefore, $m(A) = m(C)$
or $\hat{\kappa_A}(\omega) = \hat{\kappa_C}(\omega)$.

By Theorem~\ref{th:LamInt}, $\Lambda = \Lambda_1\cap\Lambda_2$ is the
endpoint set of a fenestration of $\mathbb{N}$.  Let $\lambda$ be any
element of $\Lambda$. Then
\begin{eqnarray*}
\lambda\in\Lambda_1 &\implies& \kappa_A(\lambda) = \kappa_B(\lambda), \\
\text{and\ \quad }
\lambda\in\Lambda_2 &\implies& \kappa_B(\lambda) = \kappa_C(\lambda).
\end{eqnarray*}
Therefore, $\kappa_A(\lambda) = \kappa_C(\lambda)$. Thus, $A$ and $C$
are isobaric under the fenestration determined by $\Lambda$. Consequently,
isobary is transitive.  Since isobary is obviously reflexive and symmetric,
it is an equivalence relation.
\end{proof}


\addcontentsline{toc}{part}
{Part~II: Countable Sets}
\part*{Part~II: Countable Sets}
\label{sec:part2}

In Part~I, we focussed on subsets of $\N$. 
We will now present a more general definition, for magnums of sets
larger than $\N$ and for other sets that are not contained in $\N$.
All magnums are evaluated relative to some countable reference set,
which we normally denote by $R$. We assume $R$ is enumerated by a
monotone sequence $(r_n)_n$ and every subset $A\subset R$ inherits
an ordering from this sequence.



\section{Magnums of Countable Sets}
\label{sec:largersets}

In Part~I,  magnums were defined by (\ref{eq:MagDef}), which states that
$m(A): = \hat{\kappa_A}(\omega)$.  We can describe $m(A)$ as the
\emph{magnum of $A$ relative to $\N$} or the \emph{magnum of $A$ in $\N$}
and, to indicate this explicitly, denote it by $m(A|\N)$.
Since $A\subset\N$ implies $\kappa_A(n)<n$ for all $n$, we have, by extension,
$\hat{\kappa_A}(\omega)<\omega$, so that $m(A|\N)<\omega$ for all $A\subset\N$.

The theory presented in Part~I will now be developed
for application to a broader range of countable sets.
We will consider sets with larger magnums and
also sets that are not in $\mathscr{P}(\N)$.
In \S\ref{sec:RelMagOnPN} we define relative magnums for sets of natural numbers.
In \S\ref{sec:RelMagOnR} we define magnums relative to a general countable reference set $R$.
In \S\ref{sec:CartProducts} we determine the magnums of Cartesian products of sets.
In \ref{sec:SizeOfZ}, the magnum of the set $\Z$ of integers is determined.


\subsection{Relative Magnums for Subsets of $\N$}
\label{sec:RelMagOnPN}

For any two sets $A$ and $B$ in $\mathscr{P}(\N)$,
expressed as monotone sequences $(a_n)_n$ and $(b_n)_n$,
we define the characteristic sequence for $B$ \emph{relative to} $A$ as
$$
\chi_{B|A}(n) = 
\begin{cases}
1 & \text{if\ \ } a_n\in B \\
0 & \text{if\ \ } a_n\not\in B ,
\end{cases}
$$
and the counting sequence for $B$ \emph{relative to} $A$ as
$$
\kappa_{B|A}(1) = \chi_{B|A}(1) , \qquad\qquad
\kappa_{B|A}(n+1) = \kappa_{B|A}(n) + \chi_{B|A}(n+1) .
$$
We can now define the magnum of $B$ in $A$:
\begin{definition}
Let $A$ and $B$ be in $\mathscr{P}(\N)$.
Then, denoting the magnum of $A$ by $\hat{\kappa_{A|\N}}(\omega) = \theta_A$, 
\emph{the magnum of $B$ in $A$ is }
\begin{equation}
\boxed{
m(B|A) := \widehat{\kappa_{B|A}}(\theta_A) .
}  
\label{eq:mabdef}
\end{equation}
\end{definition}
Equation~\ref{eq:mabdef} reduces to the original
definition (\ref{eq:MagDef}) when the reference set is $\N$.
It implies 
\begin{equation}
m(B|A) = m(A\cap B|A),
\end{equation}
that is, elements of $B$ that are not in
the reference set $A$ \emph{do not count!} Thus, for example,
$$
m(2\N-1|2\N) = 0 \qquad\text{and}\qquad m(2\N | 2\N-1) = 0.
$$
\COMMENT{  
\textbf{Remark:} Perhaps we should rename $m_A(B)$ to be
\emph{``The magnum of $B$ in $A$''} rather than
\emph{``The magnum of $B$ relative to $A$''.}
\textbf{Remark:} The above definition seems to work fine if $A\subset B$
or if one or both of the sets $A$ and $B$ are finite. Also,
it implies results such as
$$
m(\N-\tfrac{1}{2}|\N) = 0 \qquad\text{and}\qquad m(\tfrac{1}{2}\N|\N) = \omega 
$$
and
$$
m(\N|\N-\tfrac{1}{2}) = 0 \qquad\text{and}\qquad m(\N|\tfrac{1}{2}\N) = \omega .
$$
}  
\textbf{Example:}
We now show that the magnum of the set of even numbers, $2\N$, in the set of
squares, $\N^{(2)}$, is equal to the magnum of the set of squares in the set of even numbers:
$$
m(2\N|\N^{(2)}) = m(\N^{(2)}|2\N)  = \frac{\sqrt{\omega}}{2} .
$$
The magnums $m(2\N|\N) = \omega/2$ and $m(\N^{2}|\N) = \sqrt{\omega}$
are already known.

(\emph{i}) \emph{Assuming the Reference Set is $A = 2\N$},
let us consider the magnum $m(B|A) = m(\N^{(2)}|2\N)$. 
The counting sequence for $B$ in $A$ is
$$
\kappa_{B|A}(n) = \left\lfloor \sqrt{\frac{n}{2}} \right\rfloor .
$$
Noting that $\theta_A = \hat{\kappa_{A|\N}}(\omega) = \omega/2$,
the magnum of $B$ relative to $A$ is given by (\ref{eq:mabdef}):
$$
m(B|A) = \widehat{\kappa_{B|A}}\left( {\omega}/{2} \right)
         = \sqrt{\frac{\omega/2}{2}} = \frac{\sqrt{\omega}}{2} .
$$

(\emph{ii}) \emph{Assuming the Reference Set is $B = \N^{(2)}$},
we consider $m(A|B) = m(2\N|\N^{(2)})$. We have
$$
\kappa_{A|B}(n) = \left\lfloor \frac{n}{2} \right\rfloor .
$$
(because every second square is even).
Now, noting that the magnum of $B$ is $\theta_B = \hat{\kappa_{B|\N}}(\omega) = \sqrt{\omega}$,
we have the magnum of $A$ relative to $B$:
$$
m(A|B) = \widehat{\kappa_{A|B}}(\sqrt{\omega}) = \frac{\sqrt{\omega}}{2} .
$$
We conclude that $m(2\N|\N^{(2)}) = m(\N^{(2)}|2\N)$.
We will prove below that this symmetry holds more generally.



\subsection{Magnums Relative to a General Reference Set $R$}
\label{sec:RelMagOnR}

In \S\ref{sec:theorems} (p.~\pageref{pg:chikap}) we defined the characteristic
sequence $\chi_A(n)$ and counting sequence $\kappa_A(n)$ of a set $A$ of natural numbers.
We generalize these sequences now, introducing the corresponding window functions $X$ and $K$.%
\footnote{$X$ and $K$ are upper-case $\chi$ and $\kappa$, or ``big-chi'' and ``big-kappa''.}
First, the definition of fenestrations is broadened to sets other than $\N$.
Then, magnums relative to arbitrary countable sets are introduced.


\subsubsection*{Fenestrations of Countable Sets}

We recall from Definition~\ref{def:genfen} (pg.~\pageref{pg:deffen}) that a
fenestration of $\N$ is a partition of $\N$ determined by a
strictly increasing function $(\Lambda(n))_n$
whose endpoint set $\Lambda$ is an omega set: $\Lambda\in\BOM$.
We now generalize this idea to arbitrary countable sets $R$.
Let $\Lambda:\N\to\N$ be a strictly increasing map and let $\Lambda(0)=0$.
If we define $N_n = (\Lambda(n-1),\Lambda(n)]_{\N}$, then
the family $\mathscr{N} = \{ N_n : n\in\N \}$ is a partition of $\N$,
with $\N = \biguplus_{n\in\N} N_n$.  The length of the $n$-th component
is $L_n = \Lambda(n)-\Lambda(n-1)$.
The extension of $\Lambda:\N\to\N$ to $\hat\Lambda:\Nn\to\Nn$ yields the
extension of $N_n$ to $\hat{N}(\nu)$.

The partition $\mathscr{N} = \{ N_n : n\in\N \}$ of $\N$ is a \emph{fenestration}
of $\N$ if $\omega$ is in the endpoint set of the extension; that is,
if there exists $\nu\in\Nn$ such that $\hat{\Lambda}(\nu) = \omega$.

We now let $R$ be an arbitrary countably infinite set with enumeration
$R = \{ r_n : n\in\N \}$. Let the magnum of $R$  be denoted by $\theta$
(this may be known or may remain to be determined).
The partition $\mathscr{N} = \{ N_n : n\in\N \}$ of $\N$ 
determines a partition of $R$: if we define $R_n = \{ r_k : k\in N_n \}$,
then  $R = \biguplus_{n\in\N} R_{n}$.
\begin{definition}
Let $R=(r_n)_n$ be a sequence with magnum $m(R) = \theta$.
The partition $\mathscr{R} = \{ R_n : n\in\N \}$ of $R$ is a
\emph{fenestration} of $R$ if $\theta$ is in the endpoint set
of the extension of $\Lambda$; that is, if there exists
$\nu\in\Nn$ such that  $\hat{\Lambda}(\nu) = \theta$.
\end{definition}


\subsubsection*{Window Functions $X$ and $K$}
\label{sec:X-and-K}

So far, we have used $\N$ with its canonical ordering as the reference set, 
to which other sets are compared. For clarity, we might have written $m(A|\N)$
everywhere we wrote $m(A)$.  We now consider an arbitrary countable
reference set $R$, with an enumeration $R = \{ r_n : n\in\N\}$ which 
determines an ordering $r_m\prec r_n \iff m<n$.  We assume that the magnum
of $R$ is $m(R)=\theta$. Let $\mathscr{R} = \{R_n : n\in\N\}$ be a fenestration
of $R$ with endpoint set $\Lambda$ and window lengths $L_n = \Lambda(n) - \Lambda(n-1)$.

Now we consider a countable set $A = \{a_n:n\in\N\}$ that is a subset of $R$.
For any $a_n\in A$ there is an element $r_m\in R$ such that $a_n = r_m$ and the
ordering of $R$ determines an ordering of $A$. The set $A$ may be partitioned into
discrete, finite components $A_n = A\cap R_n$ such that $A = \biguplus A_n$.
The indicator sequence and counting sequence of $A$ are $(X_A(n))_n$ and
$(K_A(n))_n$, where
$$
X_A(n) = |A_n| \qquad\text{and}\qquad K_A(n) = \sum_{k=1}^n X_A(k) . 
$$

\subsubsection*{Magnums Relative to $R$}

We now define the relative magnum of the set $A\subset R$:
\begin{definition} \label{def:relmag}
Let $R=(r_n)_n$ be a set with magnum $m(R) = \theta$ and let $A$ be a subset of $R$.
If $\mathscr{R}$ is a fenestration of $R$ with partition function $\Lambda$,
the {magnum of $A$ relative to $R$} is given by
\begin{equation}
\boxed{
m(A|R) := \hat{K_A}(\Lambdahatinv(\theta)) .
\label{eq:mra}
}  
\end{equation}
\end{definition}
It is clear that the sequence $(K_A(n))_n$ depends on the ordering of $R$. 
As a result, the magnum of $A$ relative to $R$ may change if the ordering
of $R$ changes. In the case where all elements of the fenestration
are singletons, $\Lambda(n) = n$ and 
$m(A|R) = \hat{K_A}(\theta) = \hat{\kappa_A}(\theta)$.

We may interpret (\ref{eq:mra}) in two ways: (i) if $\theta = m(R)$ is known,
the equation provides the magnum of $A$ relative to $R$; (ii) if the magnum of $A$
is already known, (\ref{eq:mra}) constrains the value of $\theta$, the
magnum of $R$. However, $\theta$ may not be determined uniquely by (\ref{eq:mra}).%
\footnote{
The partition function $\Lambda$ determines the values of the counting sequence
$K_R(n) = \Lambda(n)$.  Therefore, by (\ref{eq:mra}), \emph{the magnum of $R$
relative to itself} is $m(R|R) = \theta$.}

We now show that the magnum of $A$ relative to $R$ given in Definition~\ref{def:relmag}
does not depend on the fenestration of $R$ or on the ordering of $A$.
\begin{theorem}
(a) The magnum of $A\subset R$ is independent of the fenestration of $R$.
(b) If set $A$ is reordered to $\tilde{A}$, its magnum does not change:
$m(\tilde{A}|R) = m(A|R)$.  
\label{th:R-and-A}
\end{theorem}
\begin{proof}
(a) Let $m(R)=\theta$ and $\Lambda_1$ and $\Lambda_2$ be two partition
functions on $\N$ that determine fenestrations of $R$,
namely $(R^i_n)_n=\{\Lambda_i(n-1)+1,\ldots, \Lambda_i(n)\}$ for $i\in\{1,2\}$,
such that $\hat\Lambda_1(\nu_1)=\hat\Lambda_2(\nu_2)=\theta$. 
By Theorem~\ref{th:LamInt}, applied with $\theta$ in place of $\omega$,
we have a fenestration $\mathscr{R}$ of $R$ with endpoint set
$\Lambda = \Lambda_1 \cap \Lambda_2$. 
Moreover $\theta=\hat\Lambda(\nu)$ for some $\nu\in\Nn$. 

We may write $\Lambda(n) = \Lambda_1(f(n))=\Lambda_2(g(n))$ for increasing functions
$f,g:\N\to\N$.  We let $f(0)=g(0)=0$.   It follows $\hat\Lambda(\mu)=\hat\Lambda_1(\hat
f(\mu))=\hat\Lambda_2(\hat g(\mu))$ for all $\mu\in\Nn$ and in
particular $\theta=\hat\Lambda(\nu)=\hat\Lambda_1(\hat
f(\nu))=\hat\Lambda_2(\hat g(\nu))$.  Thus $\nu_1= \hat f(\nu)$ and
$\nu_2=\hat g(\nu)$.

Since $R_n=\{\Lambda(n-1)+1,\ldots,\Lambda(n)\} =
\{\Lambda_1(f(n-1))+1,\dots,\Lambda_1(f(n))\}=R^1_{f(n-1)+1}\cup
R^1_{f(n-1)+2}\cup\cdots\cup R^1_{f(n)}$ we have $X_A(n)=|A\cap R_n|=|A\cap
R^1_{f(n-1)+1}|+\cdots + |A\cap R^1_{f(n)}|$ and $K_A(n)=\sum_{k=1}^n
X_A(k)= \sum_{k=1}^{f(n)}X^1_A(k)=K^1_A(f(n))$.   Similarly
$K_A(n)=K^2_A(g(n))$.  Extending this to $\Nn$ we have that
$\hat{K_A}(\Lambdahatinv(\theta)) = \hat{K_A}(\nu) = 
\hat{K^1_A}(\hat{f}(\nu)) = \hat{K^1_A}(\nu_1) =
\hat{K^1_A}(\hat \Lambda_1^{-1}(\theta))$.
Similarly, $\hat{K_A}(\Lambdahatinv(\theta)) = \hat{K^2_A}(\hat \Lambda_2^{-1}(\theta))$,
so that $m(A|R)$ is independent of the fenestration of $R$.

(b) 
Let $\tilde{A}$ be a reordering of $A$. The independence of the magnum of $A$ on its
ordering follows immediately from the fact that
$X_{\tilde{A}}(n)=|\tilde{A}_n|=|\tilde{A}\cap R_n|=|A\cap R_n|=X_A(n)$,
so that $K_{\tilde{A}}(n) = K_A(n)$.  In particular, setting $A=R$ we see that the magnum
of $R$ relative to itself remains unchanged under a reordering: $m(\tilde{R}|R)=m(R|R)$.  
\end{proof}


\textbf{Example:}
We consider the magnum of $\N$ relative to $R = \half\N$, the set of positive half-integers:
\begin{equation}
\half\N =  \{\tfrac{1}{2}, 1,\tfrac{3}{2}, 2,\tfrac{5}{2}, 3,\tfrac{7}{2}, 4, \tfrac{9}{2}, 5,\dots \}.
\nonumber
\end{equation}
Choosing $R=\half\N$ as a reference set, we consider the magnum of its
subset $\N$ relative to $R$.
We denote the magnum of $R$ by $\theta$ and fenestrate it by doubleton windows
$R_n = \{ n-\half, n \}$ so that $\Lambda(n) = 2n$ and $\Lambda^{-1}(n) = n/2$. 
The indicator sequence of $\N$ is $X_{\N}(n) = | \N\cap R_n | = 1$ and
the counting sequence is $K_{\N}(n) = n$.  Therefore, using (\ref{eq:mra}),
$$
m_{R}(\N) = \widehat{K_{\N}}(\Lambdahatinv(\theta)) = \theta/2. 
$$
Thus, the relative sizes of $\N$ and $\half\N$ are exactly as expected.

To illustrate the sensitivity to the ordering of the reference set, let us consider
\begin{equation}
\widetilde{R} = \half\widetilde{\N} = 
 \{ \tfrac{1}{2},\tfrac{3}{2},1, \tfrac{5}{2},\tfrac{7}{2}, 2, \tfrac{9}{2},\tfrac{11}{2}, 3, \dots \}.
\nonumber
\end{equation}
as the reference set (where the elements are in ascending order).
We denote the magnum of $\tilde{R}$ by $\theta$ and fenestrate it by windows
$R_n = \{\tfrac{4n-3}{2},\tfrac{4n-1}{2},n\}$ so that $\Lambda(n) = 3n$ and $\Lambda^{-1}(n) = n/3$. 
The counting sequence is again $K_{\N}(n) = n$.  Then, using (\ref{eq:mra}), 
the magnum of $\N$ relative to $\tilde{R}$ is $\theta/3$.


\subsubsection*{Symmetry of the Relative Magnums}

Let $R$ be an ordered countable set with magnum $\theta_R$ and with subsets $A$ and $B$.
Assume that $R$ is enumerated as the sequence $R = (r_n)_{n\in\N}$
and that $A$ and $B$ inherit orderings as subsequences of $R$.
We now establish a symmetry relationship for the magnums of $A$ and $B$
relative to each other:
\begin{theorem} The magnum of $A$ in $B$ is equal to the magnum of $B$ in $A$:
\label{th:symmAB}
$$
m(A|B) = m(B|A) .
$$
\end{theorem}
\begin{proof}
For any $n\in\N$, we define $R_n = \{r_1, r_2, \dots , r_n \}$.
If we let $\kappa_{A|R}(n) = | A\cap R_n| = k$, then $a_k \le r_n$ and $a_{k+1} > r_n$.
In a similar manner, if $\kappa_{B|R}(n) = | B\cap R_n| = \ell$, then $b_\ell \le r_n$ and $b_{\ell+1} > r_n$.
It follows that
\begin{eqnarray*}
\kappa_{B|A}(\kappa_{A|R}(n)) = \kappa_{B|A}(k) = | A_k\cap B | , \\
\kappa_{A|B}(\kappa_{B|R}(n)) = \kappa_{A|B}(\ell) = | A\cap B_\ell | .
\end{eqnarray*}
However, noting that $|A_k\cap B| = |A_k\cap B_\ell| = |A\cap B_\ell|$, this implies
$\kappa_{B|A}(\kappa_{A|R}(n)) = \kappa_{A|B}(\kappa_{B|R}(n))$. By the Axiom of Extension,
$\hat{\kappa_{B|A}}(\hat{\kappa_{A|R}}(\nu)) = \hat{\kappa_{A|B}}(\hat{\kappa_{B|R}}(\nu))$ 
for $\nu\in\Nn$. We apply this for $\nu = \theta_R$, noting that
$\hat{\kappa_{A|R}}(\theta_R) = \theta_A$ and $\hat{\kappa_{B|R}}(\theta_R) = \theta_B$,
to obtain 
$$
\hat{\kappa_{B|A}}(\theta_A) = \hat{\kappa_{A|B}}(\theta_B)
\qquad\text{or}\qquad m(B|A) = m(A|B), 
$$
which is the required symmetry result.
\end{proof}


\subsubsection*{Theorems for Relative Magnums}

All the theorems of \S\ref{sec:SomeTheorems} can be generalised by replacing
the reference set $\N$ by an arbitrary ordered countable set $R$ with magnum $\theta$.
The proofs follow with only minor modifications. For example, a partition
$\{R_n:n\in\N\}$ of $R$ induces partitions $\{A_n:n\in\N\}$ and $\{B_n:n\in\N\}$
of subsets $A$ and $B$, where $A_n = A\cap R_n$ and $B_n = B\cap R_n$.
Then $A$ and $B$ are isobaric in $R$ if $|A_n| =  |B_n|$ for all $n$,
and isobary is an equivalence relation on $\mathscr{P}(R)$.
In a similar manner, the theorems in \S\ref{sec:M-theorems} generalise for
magnums relative to an arbitrary ordered countable set $R$.
For an example, see Theorem~\ref{th:SurrDen} below.

\subsubsection*{Surreal Density}

We now broaden the density theorem
(Theorem~\ref{th:DensityTheorem}) to include density relative to a
countable reference set $R$.  By
analogy with the density sequence $\rho_A(n)=\kappa_A(n)/n=
\kappa_A(n)/\kappa_{\N}(n)$, with limit $\rho(A)$ (where it exists),
we define the relative density sequence \[\rho_{A|R}(k)= K_A(k)/K_R(k)\]
and the relative density $\rho(A|R)$ to be the limit of this sequence
where it exists.  If we now define the \emph{surreal density} of a set
$A\subseteq R$ as
$$
{\sigma}_R(A) := \frac{\hat{K_A}(\Lambdahatinv
  (\theta))}{\hat{K_{R}}(\Lambdahatinv (\theta))}
           = \frac{m(A|R)}{\theta},
$$
we can state and prove the Surreal Density Theorem:
\begin{theorem}\label{th:SurrDen}
  Let the reference set $R$ have magnum $\theta$ and $A\subset R$ have
  relative density $\rho_R(A) \in [0,1]$.  Then   $m(A|R)\approx \rho(A|R)\theta$.
\end{theorem}
\begin{proof}
We apply Lemma~\ref{lemma:ratio} (\S\ref{sec:DenTheorem}) with $f=K_A$ and $g=K_R$,
where $\nu=\hat\Lambda^{-1}(\theta)$.  Since, by definition,
$\lim_n {f(n)}/{g(n)}=\rho(A|R)$, we find that
$$
{\sigma}(A|R) = \frac{m(A|R)}{\theta}
	      = \frac{\hat K_A(\hat \Lambda^{-1}(\theta))}
              {\hat K_R(\hat\Lambda^{-1}(\theta))}
            \approx \rho(A|R)
$$
or, equivalently, $m(A|R)=\rho(A|R)\theta + o(\theta)$.
\end{proof}

\subsubsection*{Bayes' Rule}


\begin{definition}
\label{def:sigAinB}
Let $A$ and $B$ be two subsets of a reference set $R$.%
\footnote{We do not assume that either $A\subset B$ or $B\subset A$ necessarily holds.}
The surreal density of $A$ in $B$ is defined as 
\begin{equation}
\boxed{
\sigma(A|B) := \frac{m(A\cap B|R)}{m(B|R)}
}  
\label{eq:surrdensdef}
\end{equation}
\end{definition}
We assume that the magnum of $R$ is $m(R|R) = \theta$. Then 
$$
\sigma(A\cap B|R) = \frac{m(A\cap B|R)}{\theta}
\qquad\text{and}\qquad
\sigma(B|R) = \frac{m(B|R)}{\theta} ,
$$
from which we immediately deduce 
$$
\sigma(A|B) = \frac{\sigma(A\cap B|R)}{\sigma(B|R)} .
$$
Noting a similar relationship with $A$ and $B$ exchanged, we have
$$
\sigma(A\cap B|R) =  \sigma(A|B)\sigma(B|R)
\qquad\text{and}\qquad
\sigma(A\cap B|R) =  \sigma(B|A)\sigma(A|R) .
$$
Equating the right hand sides of these equations yields
\begin{equation}
\boxed{
\sigma(A|B) =\frac{\sigma(A|R)}{\sigma(B|R)} \sigma(B|A) .
}  
\label{eq:surrBayes}
\end{equation}
Now, interpreting the densities as probabilities on a sample space $R$,
we change notation as follows:
for probabilities, $\sigma(A|R) \to P(A)$  and $\sigma(B|R) \to P(B)$
and, for conditional probabilities,
$\sigma(A|B) \to P(A|B)$ and $\sigma(B|A) \to P(B|A)$, so that
(\ref{eq:surrBayes}) becomes
$$
P(A|B) = \frac{P(A)}{P(B)} P(B|A) .
$$
which is a standard expression for conditional probability often
referred to as Bayes' Rule. Thus, (\ref{eq:surrBayes}) is an 
extension of Bayes' Rule to the surnatural domain.

\textbf{Example.} 
Suppose $A = 2\N$, $B = 3\N$ and $R = \N$.
So, $A\subset R$ and $B\subset R$ but $A\not\subset B$ and $B\not\subset A$. 
The magnums are $m(A|R) = \tfrac{\omega}{2}$, $m(B|R) = \tfrac{\omega}{3}$ and $m(R|R) = \omega$.
Also, $A\cap B = 6\N$ so $m(A\cap B|R) = \tfrac{\omega}{6}$.
From Definition~\ref{def:sigAinB} we have 
$\sigma(A|R) = \frac{1}{2}$, $\sigma(B|R) = \frac{1}{3}$, 
$\sigma(B|A) = \frac{1}{3}$ and $\sigma(A|B) = \frac{1}{2}$,
so that (\ref{eq:surrBayes}) is satisfied.


\subsection{Cartesian Products of Sets}
\label{sec:CartProducts}
    
It is necessary to distinguish between ``the union of disjoint sets''
and ``the disjoint union of sets''.  We have denoted the union of disjoint sets
$A$ and $B$ by $A\uplus B$.  This is essentially a mnemonic device, reminding us
that the sets are disjoint.  But it is useful to consider the disjoint union
of sets that may not be disjoint. Thus, we define
\begin{equation}
A \sqcup B := (A\times\{1\}) \uplus (B\times\{2\}) \,.
\label{eq:disjointUnion}
\end{equation}
We observe that neither set $A$ nor set $B$ is a subset of $A\sqcup B$.

If we define $\kappa_{A\times\{1\}}(n) = |\, (A\times\{1\}) \cap (I_n \times\{1\})\, |$,
where, as usual, $I_n = \{1, 2, \dots, n\}$, 
then the counting sequences of $A$ and $A\times\{1\}$ are identical:
$\kappa(A\times\{1\}) = \kappa(A)$. Moreover, as $A\times\{1\}$ and $B\times\{2\}$
are disjoint, $\kappa(A\sqcup B) = \kappa(A) + \kappa(B)$. We will see that
many consequences, such as $m(A\sqcup B) = m(A)+m(B)$, $m(\N\sqcup I_n)
= \omega + n$ and $m(\N\sqcup\N) = 2\omega$, follow from these results.

\subsubsection*{Square Ordering}

For Cartesian products of ordered sets $U$ and $V$, it is convenient to introduce
a \emph{square ordering} of $W = U\times V$.
Let $\{U_m:m\in\N\}$ and $\{V_n:n\in\N\}$ be partitions of $U$ and $V$ into
finite components.  We assume that the components are ordered so that
$(k_1\in U_{m_1}) \land (k_2\in U_{m_2}) \land (m_1 < m_2) \implies (k_1 < k_2)$,
and similarly for $V$.  We define the $n$-th component of $U\times V$ as the
union of $U_k\times V_\ell$ where the maximum of $\{k,\ell\}$ equals $n$. Thus,
$$
W_n = \biguplus_{k,\ell} \{ U_k\times V_\ell : \max\{k,\ell\}=n \} 
\qquad\text{and}\qquad W = U\times V = \biguplus_{n} W_n .
$$
The indicator sequence of $W$ is $X_W(n) = |W_n|$ and the counting
sequence is $K_W(n) = \sum_{k=1}^n X_W(k)$.


\begin{theorem} 
\label{th:SqCup}
Let $U$ and $V$ be two sets with counting sequences
$K(U)$ and $K(V)$.%
\footnote{We assume that $U$ and $V$ are subsets of the same reference
  set $R$.  See Remarks in this subsection.}
Then the counting sequence of $U\sqcup V$ is
$K(U\sqcup V) = K(U) + K(V)$ and the magnum of the disjoint union
is $m(U\sqcup V) = m(U) + m(V)$.
\end{theorem}
\begin{proof}
The counting sequences of $U$ and $U\times\{1\}$ are identical, as are
those of $V$ and $V\times\{2\}$. Thus,
$K(U\times\{1\}) = K(U)$ and $K(V\times\{2\}) = K(V)$.
Moreover, as $U\times\{1\}$ and $V\times\{2\}$ are disjoint,
$K(U\sqcup V) = K(U) + K(V)$. Extending the domains to $\Nn$,
we find that $m(U\sqcup V) = m(U) + m(V)$.
\end{proof}
\begin{theorem} 
\label{th:CarPro}
Let $U$ and $V$ be two sets with counting sequences
$K(U)$ and $K(V)$.  Then the counting sequence of $W = U\times V$
is $K(U\times V) = K(U) \cdot K(V)$ and the magnum of the
product is $m(U\times V) = m(U)\cdot m(V)$.
\end{theorem}
\begin{proof}
It is simple to show that
$$
| W_m| = \left| \biguplus_{\max{\{k,\ell\}}=m } U_k\times V_\ell \right| 
       = \sum_{ \max{\{k,\ell\}}=m } |U_k\times V_\ell| 
       = \sum_{ \max{\{k,\ell\}}=m } |U_k|\cdot |V_\ell| .
$$
Then, summing from $1$ to $n$ yields
$$
K_W(n) = \sum_{m\le n} | W_m | 
       = \left| \biguplus_{\max{\{k,\ell\}}\le n } U_k\times V_\ell \right| 
       = \sum_{k,\ell\le n}  | U_k\cdot V_\ell | 
       = \sum_{k\le n} | U_k | \cdot \sum_{\ell\le n} | V_\ell |
       = K_U(n)\cdot K_V(n) .
$$
which, upon extension to $\Nn$, implies $m(U\times V) = m(U)\cdot m(V)$. 
\end{proof}


{\bf Remarks.} 
{\em
\emph{(a)}
In \S\ref{sec:X-and-K} we described how the magnum of a set is \emph{relative}
to that of a given enumerated reference set $R$ of which it is a subset.
For subsets of the natural numbers, the reference set is naturally
taken to be $\N$ itself.  For countable sets not contained in $\N$,
there may not be a canonical choice of reference set,
or its enumeration.  This matter will concern us further in
\S\ref{sec:rationals} below. 

\emph{(b)}
The result in Theorem~\ref{th:SqCup} may be stated 
more formally as $m(U\sqcup V|R\sqcup R)= m(U|R) + m(V|R)$.
The validity of this
statement relies on the fact that if $m(R)=\theta$ and $\mathscr{R}=(R_n)_n$
is a fenestration of $R$ with endpoint set $\Lambda$ such that
$\hat\Lambda(\nu)=\theta$, then $R\sqcup
R=\{r_1\times\{1\},r_1\times\{2\},r_2\times\{1\},r_2\times\{2\},\ldots\}$ has
fenestration $(R\sqcup R)_n= R_n\sqcup R_n$ with partition function
$2\Lambda$, and $2\hat\Lambda(\nu)=2\theta$ is the magnum of the
reference set $R\sqcup R$ for $U\sqcup V$.

\emph{(c)}
Analogously, Theorem~\ref{th:CarPro}
may be expressed as $m(U\times V|R\times R)=m(U|R)m(V|R)$ where the
reference set $R\times R$ is square ordered.    Here, the fenestration
of $R\times R$ is determined by
$(R\times R)_n = \biguplus_{\max\{k,\ell\}=n} R_k\times R_\ell$,
and satisfies $|(R\times R)_n| = (\Lambda(n))^2-(\Lambda(n-1))^2$.
The partition function $\Lambda^2$ satisfies $\hat{\Lambda^2}(\nu)= \theta^2$
and this is the magnum of the reference set $R\times R$. 
The count of elements in the proofs of
Theorems~\ref{th:SqCup} and~\ref{th:CarPro} is now valid with respect
to these reference sets whose magnums have been determined.  In particular, if $U$
and $V$ are subsets of $\N$ then $m_{\N\times\N}(U\times V)=
m(U)m(V)$ and, most fundamentally,
$m_{\N\times\N}(\N\times\N)=\omega^2$.
}

\medskip
Unless otherwise stated, we assume that $\N\times\N$ has square ordering, with
components of its fenestration beginning with
$$
W_1 = \{(1,1)\}, \ \ 
W_2 = \{(1,2),(2,2),(2,1)\}, \ \ 
W_3 = \{(1,3),(2,3),(3,3),(3,2),(3,1)\}, \ \  \dots\ .
$$
The indicator sequence and counting sequence of $\N^2$ are $X(n) = 2n-1$
and $K(n) = 1 + 3 + \dots + (2n-1) = n^2$.  Extending the domain to $\Nn$ we have
$\hat{K}(\nu) = \nu^2$ and the magnum is $m(\N^2) = \hat{K}(\omega) = \omega^2$.
Further results such as $m(\N^k) = \omega^k$ and $m(\N^k\sqcup\N^\ell) = \omega^k+\omega^\ell$
follow in like manner.


\newcommand{\dHat}[1]{\Hat{\Hat{#1}}}

\subsection{Magnum of the Integers}
\label{sec:SizeOfZ}

We calculate the magnum of the set $\Z$ of integers from the counting
sequence using (\ref{eq:mra}) and also directly from the genetic form.
In both cases, $m(\Z) = 2\omega+1$.

\subsubsection*{The Magnum of $\Z$ from the Counting Sequence}

We choose $R = \Z$ as the reference set. Assume $m(R) = \theta$ (to be determined). 
The partition function $\Lambda(n) = 2n-1$ determines a partition of $\N$
(we remark that this is not a fenestration of $\N$).  The inverse map is
$\Lambda^{-1}(n) = (n+1)/2$.  We order $R$ in the following way:
$$
R = (r_k)_k = ( 0 \prec 1 \prec -1 \prec 2 \prec -2 \prec 3 \prec -3 
                  \prec 4 \prec -4 \prec  \dots ).
$$
The set $R$ inherits a partition from that of $\N$, with
$\mathscr{R} = \{R_n\}$ where $R_n = \{n-1,-(n-1)\}$:
$$
\mathscr{R} = ( \{0\}, \{1, -1\}, \{2, -2\}, \{3, -3\}, \{4, -4\}, \dots ).
$$
The condition that $\mathscr{R}$ is a fenestration requires
$\nu = \Lambdahatinv(\theta) \in \Nn$, which implies that $\theta = 2\nu-1$ is odd.
We observe that $K_{\N}(n) = n-1$ and assume that $m(\N|R) = \omega$. Then (\ref{eq:mra}) becomes
$$
\omega = \hat K_{\N}(\Lambdahatinv(\theta)) 
       = \left( \frac{\theta+1}{2} \right) - 1
       = \frac{\theta -1}{2} ,
$$
and, solving for $\theta$, we get $m(\Z) = \theta = 2\omega+1$.

\subsubsection*{The Magnum of $\Z$ from the Genetic Form}

The integers can be partitioned in an obvious way as $\Z = \Z^{-}\uplus\{0\}\uplus\Z^{+}$.
It seems reasonable to expect that $\mathbb{Z}^{-}$ has the same magnum as $\mathbb{Z}^{+}$.
We will prove this by means of the genetic forms for $\Z^{+} = \N$ and $\Z^{-} = -\N$.
\begin{theorem}
\label{th:BooksUp}
If a countable set $A$ is magnumbered on Day~$\alpha$ ($A\in\Nscr_\alpha$), then
$A^{-}=\{-a : a\in A\}$ is also magnumbered on Day~$\alpha$ and $m(A^{-}) = m(A)$.
\end{theorem}
\begin{proof}
The result follows easily by induction on $\alpha$ using the genetic form.
It is a consequence of the symmetry of the number-tree of surreals ---
where all number-pairs $(x, -x)$ are born on the same day ---
that the sets $A$ and $A^{-}$ are magnumbered on the same day.
Thus, the genetic form
$$
m(A) = \{ m(B) : B \in \Oscr_\alpha, B \subset A \ | \ m(C) : C \in \Oscr_\alpha, A \subset C \} \,.
$$
remains valid if all sets are replaced by their reflections in $0$:
$$
m(A^{-}) = \{ m(B^{-}) : B^{-} \in \Oscr_\alpha, B^{-} \subset A^{-} \ | \ m(C^{-}) : C^{-} \in \Oscr_\alpha, A^{-} \subset C^{-} \} \,.
$$
Therefore, $m(A^{-}) = m(A)$.
\end{proof}
We remark that, while magnums are invariant under a reflection about the origin,
they are not in general invariant under reflections about other points or under translations
(except for bounded sets).
\begin{corollary}
\label{cor:ZminusZplus}
(a) The magnums of the positive integers and of the negative integers are equal:
$m(\mathbb{Z}^{-}) = m(\mathbb{Z}^{+}) = \omega$.
(b) The magnum of the set of integers is $m(\mathbb{Z}) = 2\omega + 1$.
\end{corollary}
\begin{proof}
(a) follows immediately from Theorem~\ref{th:BooksUp}.
Result (b) follows from 
$\mathbb{Z} = \mathbb{Z}^{-} \uplus \{0\} \uplus \mathbb{Z}^{+}$
and finite additivity of magnums.
\end{proof}
It is easy to show that $m(2\Z-1)=\omega$ and $m(2\Z)=\omega+1$.
Thus, the set of even integers has one more element than the set of odd ones,
which harmonizes with our intuition.

The idea that the set of negative rational numbers has the same size as
the set of positive ones is intuitively reasonable.
Both \cite{BeNa19} and \cite{Trl24} also conclude that $\Q^{-}$ and $\Q^{+}$
are the same size.
Although we have not yet established the value of $m(\Q)$,
the following corollary is an immediate consequence of Theorem~\ref{th:BooksUp}:
\begin{corollary}
\label{cor:ZminusZplus}
The magnums of the sets of positive and of negative rational numbers are equal:
$m(\mathbb{Q}^{-}) = m(\mathbb{Q}^{+})$.
\end{corollary}


\section{The Set of Rational Numbers}
\label{sec:rationals}


The set of  positive rational numbers may be defined as
$$
\Q^{+} = \left\{\frac{m}{n} : m,n\in\N \land \gcd(m,n) = 1 \right\}.
$$
There is an obvious bijection between the set $\Q^{+}_{<1}$ of positive rationals
less than $1$ and the set $\Q^{+}_{>1}$ of rationals greater than $1$:
the map $f : q \mapsto 1/q$ sends ${m}/{n}$ to ${n}/{m}$ for all $m<n\in\N$.
In Cantorian terms, the two sets are `the same size', but this
clashes with our intuition that $\Q^{+}_{>1}$ must be vastly
greater than $\Q^{+}_{<1}$. Bolzano, in his book 
\emph{Paradoxes of the Infinite} \cite{Bol15}, argued persuasively that
the size of a [bounded] set should be invariant under translation.
In particular, the magnitude of a set of rational numbers in an interval
should depend only on the length of the interval.
To achieve this, we will reorder the set $\Q$ and prove that, for the new ordering,
the magnum is homogeneous; that is, the magnum of the unit rational
interval $Q_k = (k-1,k]_\Q$ is independent of $k$.

\begin{figure}[h]
\begin{center}
\includegraphics[width=0.6\textwidth]{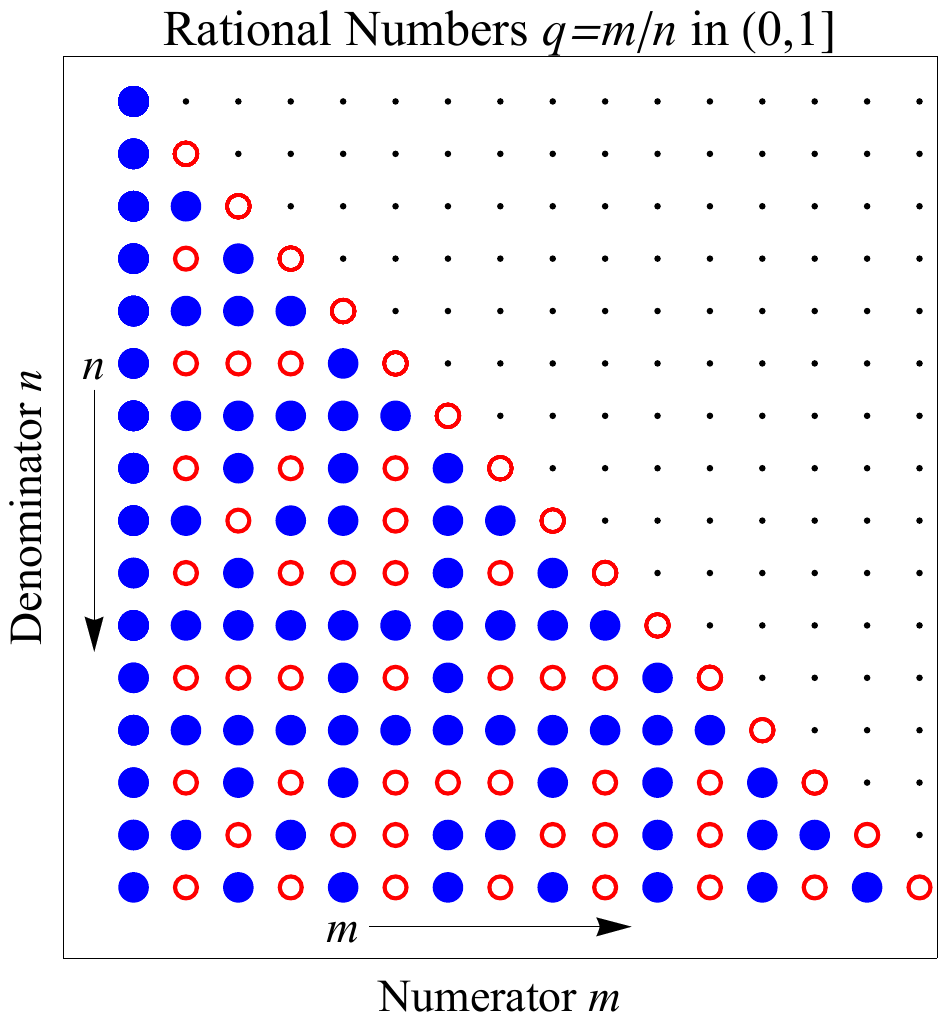}
\caption{The positive rational numbers (small black dots), 
         rational numbers in $Q_1 = (0,1]_\Q$ with duplications (open circles, red online)
         and rational numbers in $Q_1$ in reduced form (large dots, blue online).}
\label{fig:rationals}
\end{center}
\end{figure}

\subsection{The Size of $\Q$: Previous Estimates}

Several earlier studies have considered the size of the set of rational numbers,
obtaining disparate results.  In her review \cite{Wen24}, Wenmackers compared six
methods of measuring set size. We will focus here on two recent publications,
those of Trlifajov\'{a} \cite{Trl24} and of Benci and DiNasso \cite{BeNa19}.
As observed by Wenmackers, the axioms of these theories do not uniquely determine
the relative sizes of every pair of sets in $\mathscr{P}(\N)$.

\subsubsection*{Partitioning $\Q$ According to Trlifajov\'{a}}

A schematic diagram of the positive rational numbers $q=m/n$
is presented in Fig.~\ref{fig:rationals}. The numerator $m$ increases
towards the right and the denominator $n$ downwards.
Large dots (blue online) show the rational numbers in $Q_1 = (0,1]_\Q$ in reduced form.
The number of rationals in the $n$-th row is $\varphi(n)$, Euler's totient function,
that is, the number of positive integers not greater than and prime to $n$.
The first twelve values of this function are
$$
\varphi(n) =  ( 1, 1, 2, 2, 4, 2, 6, 4, 6, 4,10, 4 )  \,,
$$
as indicated in Fig.~\ref{fig:rationals}.
The count of rationals in the first $n$ rows is given by
$$
K_{Q_1}(n) = \Phi(n) = \sum_{k=1}^n \varphi(k) \,.
$$
$\Phi(n)$ is also the number of positive terms in the $n$-th Farey sequence
\cite[Ch.~3]{HaWr60}, and the first twelve values of this function are
$$
\Phi(n)  =  ( 1, 2, 4, 6,10,12,18,22,28,32,42,46 )   \,.
$$
Basic number theory provides a value for the relative sizes of the
set $Q_1$ of rational numbers in the interval $(0,1]$ and the set 
$\N^2$ of pairs of natural numbers $(m,n)$.  Hardy and Wright
\cite[p.~268]{HaWr60}, showed that, with the canonical (square)
ordering of $\N^2$, the counting sequence of $Q_1$ is
\label{pg:HandW}   
\begin{equation}
K_{Q_1}(n) = \Phi(n) = \frac{3 n^2}{\pi^2} + O(n\log n) \,.
\label{eq:HandW}
\end{equation}
Extending $\Phi$ to $\Nn$ and assuming that the first term continues to dominate,
one may conclude that 
\begin{equation}
m_{\N^2}(Q_1) = K_{Q_1}(\omega) \approx \frac{3}{\pi^2} \omega^2 
\label{eq:mU}
\end{equation}
[$f(\omega) \approx \omega^\alpha$ means $f(\omega) = \omega^\alpha + o(\omega^\alpha)$
where $o(\omega^\alpha) / \omega^\alpha$ is infinitesimal].
To obtain the size of the rational numbers, Trlifajov\'{a} \cite{Trl24}
requires the size of a set to be invariant under translation. 
This implies that the magnum of any rational interval $[a,b]\cap\mathbb{Q}$
depends only upon $|a-b|$.
Trlifajov\'{a} \cite{Trl24} proves that (\ref{eq:mU}) remains valid for
any unit interval, confirming the translational invariance of the magnum.
However, her reasoning leads her to the \emph{not-so-obvious}
conclusion that $m(\mathbb{Q})$ is of order $\omega^3$, which, given
$m(\N^2)=\omega^2$, is inconsistent with the Euclidean Principle.

\subsubsection*{Partitioning $\Q$ According to Benci and DiNasso}

In contrast to \cite{Trl24}, Benci and DiNasso \cite{BeNa19} find that
$\mathfrak{n}_\alpha(\Q^{+}) = \Alpha^2$.%
\footnote{This is analogous to $m(\Q^{+}) = \omega^2$ in the context of the surreals.}
They discuss the inherently arbitrary nature
of numerosity, which is grounded on an underlying ultrafilter. 
Their results are  sensitive to their choice of the family of `qualified sets'.
They also present arguments that this indefiniteness has some advantages
\cite[pp.~292--293]{BeNa19}.

Benci and DiNasso \cite{BeNa19} construct a labelled system with domain $\Q$
for which the size of the set of rationals is $2\Alpha^2+1$.  They also show
that the size of a bounded interval of rational numbers is invariant under translation.
Their principal findings are summerised here:
\begin{itemize}
\item
They define sets $\bbH(n) = \{ \pm i/n : i=0,1,2,\dots,n^2 \}$ that cover the
range $[-n,+n]_\Q$ uniformly, with density $\rho \approx n$
(\emph{i.e.}, $n$ elements in each unit interval).
\item
They define the label of a rational number $q\in\Q$ by $\ell_\Q(q)
= \min\{n! : q\in \bbH(n!)\}$. We observe that 
$\ell_Q(q)$ takes values only in $\N! = \{ n!:n\in\N \}$.
\item
They note the value of the counting function:
$\varphi_X(n!) = |\{q\in X : \ell_\Q(q) \le n! \}|
=\left|\bigcup_{m\le n} (\bbH(m!)\cap X) \right|
=\left|\bbH(n!)\cap X \right| $.
The function $\varphi_X(n!)$ corresponds to our $\kappa_X(n!)$.
\item
They show that $| (k-1,k]_\Q \cap \bbH(m) | = m$ for all $m\ge k$.
This establishes the \emph{homogeneity} of the numerosity on $\Q$.
\item
They show that if $n\ge k$ then $\varphi_{(k-1,k]_\Q}(n!) = n!$,
which leads them to the result $\mathfrak{n}_{\Alpha}((k-1,k]_\Q) = \Alpha$.
All unit rational intervals have numerosity $\Alpha$.
\item
They conclude that $\mathfrak{n}_{\Alpha}(\Q) = 2\Alpha^2 + 1$.
\end{itemize}

\subsection{Square Ordering of the Set $\N\times\N$}

We consider the set $\N^2 = \N\times\N$ defined as
$$
\N\times\N = \{ (m,n) : m\in\N, n\in\N \}.
$$
If we express $(m,n)$ as an unreduced fraction $\tfrac{m}{n}$, 
we must treat $\tfrac{1}{2}$ and $\tfrac{3}{6}$ as distinct elements of $\N\times\N$.
By Theorem~\ref{th:CarPro}, the magnum of $\N\times\N$ is $\theta = \omega^2$.

\begin{figure}[h]
\begin{center}
\includegraphics[width=0.65\textwidth]{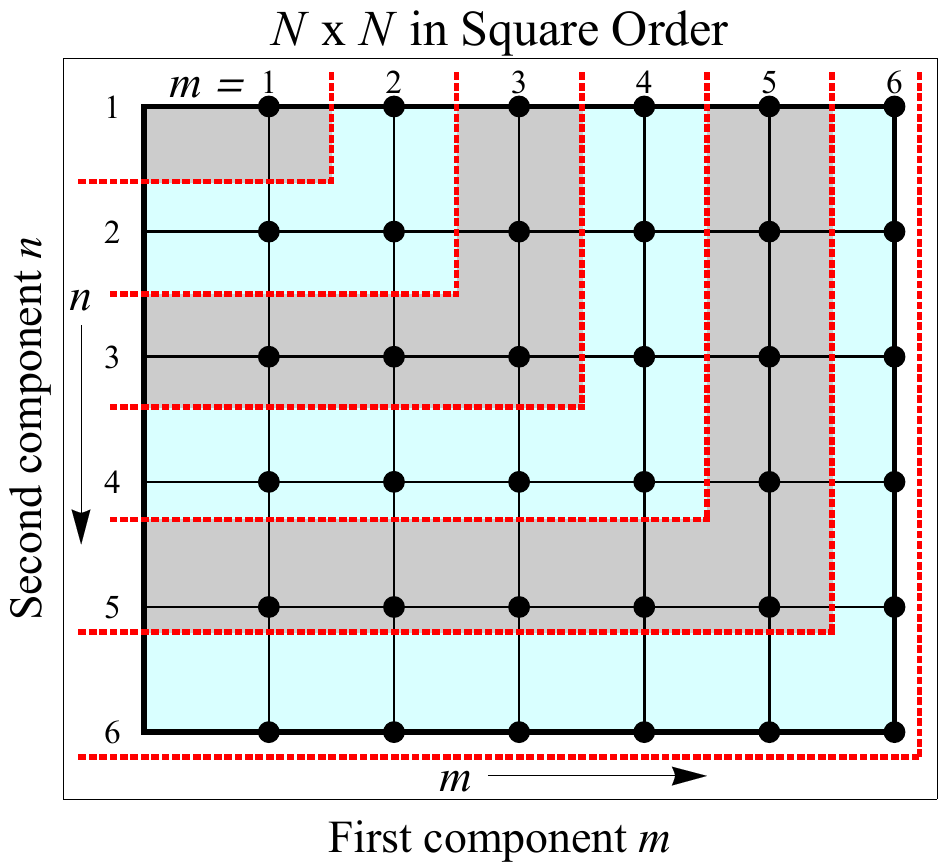}
\caption{Ordering the set $\N^2 = \N\times\N$ with the usual square arrangement.
Points represent ordered pairs $(m,n)$.
Windows (shaded) are separated by dashed lines (red online).}
\label{fig:NxN-Square}
\end{center}
\end{figure}

Fig.~\ref{fig:NxN-Square} illustrates the set $\N\times\N$ with the usual
\emph{square} ordering.  The dashed lines (red online) separate the L-shaped windows
$$
W_k = \{ (m,n) : \max(m,n) = k \}.
$$
Within each window, the elements are ordered by the magnitude of $\tfrac{m}{n}$,
starting with $\tfrac{1}{k}$ at the lower left end and finishing with $\tfrac{k}{1}$
at the upper right end of the window. The length of the window $W_k$ is $2k-1$
and the endpoint is $\Lambda(k)=k^2$.  Clearly, $\Lambda^{-1}(k) = \sqrt{k}$
and $\{\Lambda(k) : k\in\N\}$ is an omega set, so the family
$\mathscr{W} = \{ W_n : n\in\N \}$ is a fenestration.


We examine the partition $\N^2 = \N^2_{<1} \uplus \N^2_{=1} \uplus \N^2_{>1}$, where
\begin{eqnarray*} 
\N^2_{<1} &:=& \{ (m,n)\in\N^2 : m<n \} \\
\N^2_{=1} &:=& \{ (m,n)\in\N^2 : m=n \} \\
\N^2_{>1} &:=& \{ (m,n)\in\N^2 : m>n \} .
\end{eqnarray*} 
Although we have not proven the General Isobary Theorem for sets larger than $\N$,
we can easily show that the magnums $m(\N^2_{<1})$ and $m(\N^2_{>1})$ are equal:
for square ordering of $\N^2$, the counting sequences of the two sets are equal,
$K_{<1}(n) = K_{>1}(n) = n(n-1)/2$, so their magnums are 
$m(\N^2_{<1}) = m(\N^2_{>1}) = \omega(\omega-1)/2$.
Clearly, $m(\N^2_{=1}) = \omega$ so the magnums of the three disjoint sets add to $\omega^2$.


\subsubsection*{Example 2: Magnum of $\half\N$ relative to $\N^2$}

We consider the set $\half\N$ of positive half-integers. 
We note that $\half\N \subset \Q \subset \N^2$.
Since $\half\N = (\N-\half)\uplus\N$, Theorem~\ref{th:SqCup} implies that
$m(\half\N) = m(\N-\half) + m(\N)$. We seek the magnum of $\half\N$ relative
to $\N^2$ with its canonical square ordering.

The counting function of $\half\N$ is $K_{\N/2}(n) = \frac{3}{2}n$ (for $n>1$).
The fenestration of $\N^2$ is $\Lambda(n)=n^2$ so $\Lambda^{-1}(n)=\sqrt{n}$.
Since $\theta = m(\N^2) = \omega^2$, we have
$\Lambdahatinv(\omega^2)=\omega$, so (\ref{eq:mra}) implies
$m_{\N^2}(\half\N) = \frac{3}{2}\omega$.
We also note that $m_{\N^2}(\N-\half) = \frac{1}{2}\omega$.
Clearly, these values are not in harmony with our expectations.
The way out of this dilemma is to reorder the rational numbers.


\subsection{The Magnum of the Set $\Q$ of Rational Numbers}

Every positive rational number $q$ may be expressed in a unique way as an
irreducible fraction $q = {m}/{n}$, where $m$ and $n$ are coprime natural numbers
\cite{DePf17}.  We can thus represent $q$ uniquely as an ordered pair $(m,n)\in\N\times\N$,
which may be regarded as the canonical form of $q$. In this way, we may regard
$\Q^{+}$ as a proper subset of $\N\times\N$ and, in order to respect
the Euclidean Principle, we must ensure that $m(\Q^{+}) < m(\N\times\N) = \omega^2$.
Identifying $\N$ with $\N\times\{1\} = \{(n,1):n\in\N\}$, we can write $\N\subset\Q^{+}\subset\N^2$.

With square ordering of $\N^2$, the set $\Q^{+}_{<1}$ of positive rational numbers
less than $1$ has counting sequence $K_{\Q^{+}_{<1}}(n) = \Phi(n) = \sum_{k=1}^n \varphi(n)$
(see equation (\ref{eq:mU}) above).
The asymptotic result of Hardy and Wright \cite[p.~\pageref{pg:HandW}]{HaWr60}
is that $\Phi(n) \approx \chi n^2$, where $\chi = {3}/{\pi^2}$.
With square ordering of $\N^2$, the counting sequences of $\Q^{+}_{<1}$ and $\Q^{+}_{>1}$
are identical.
As a result, the density of $\Q^{+}$ relative to $\N^2$ is $2\chi$.
We may now appeal to the Density Theorem (Theorem~\ref{th:SurrDen})
to conclude that the magnum of the set of positive rational numbers is 
\begin{equation}
m_{\N^2}(\Q^{+}) = 2\chi\omega^2 + o(\omega^2)
\label{eq:QplusMag}
\end{equation}
Defining $\vartheta = m(\Q^{+})$, recalling that
$\Q = \Q^{-}\uplus\{0\}\uplus\Q^{+}$ and following an argument 
similar to that presented in \S\ref{sec:SizeOfZ} for the set $\Z$,
we may conclude that 
\begin{equation}
m(\Q) =  2 \vartheta + 1 .
\label{eq:QMag}
\end{equation}

\subsubsection*{Example 3: Parity Classes of $\Q^{+}$}

\newcommand{\QEplus}{\Q_\mathrm{E}^{+}}
\newcommand{\QOplus}{\Q_\mathrm{O}^{+}}
\newcommand{\QNplus}{\Q_\mathrm{N}^{+}}

The rationals may be partitioned into three parity classes,
\emph{even}, \emph{odd} and \emph{none}. As shown
in \cite{LyMa24}, these three classes have equal natural densities
under both Calkin--Wilf and Stern--Brocot orderings.
From the Density Theorem, it follows that 
$$
m_\Q(\QEplus) \approx m_\Q(\QOplus) \approx m_\Q(\QNplus)
              \approx \frac{\vartheta}{3}
$$
where $\vartheta$ is the magnum of $\Q^{+}$.


\newcommand{\QB}{\Q_\mathrm{B}}

\subsubsection*{Banded Ordering of the Set $\Q$}

The relationship $m(\Q_{<1}) = m(\Q_{>1})$, which follows from the 
symmetry of the counting function under the mapping $q\to 1/q$, is inconsistent with
Bolzano's criterion: 
the idea that the bounded set $\Q_{<1}$ is equal in size to the unbounded one
$\Q_{>1}$ is repugnant to our intuition. This problem was discussed
in some detail by Trlifajov\'{a} \cite{Trl24}.
Since $\N \subset \Q^{+} \subset \N^2$, we must have $\omega < m(\Q^{+}) < \omega^2$.%
\footnote{Both Trlifajov\'{a} \cite{Trl24} and Benci and DiNasso \cite{BeNa19}
obtain values for the size of $\Q^{+}$ in conflict with these limits, the former
finding a value of order $\Alpha^3$ and the latter a value equal to $\Alpha^2$.}

With the ordering of $\Q$ inherited from $\N^2$, the magnum function is not homogeneous:
the two sets $\Q^{+}_{>1}$ and $\Q^{+}_{<1}$ have equal magnums and 
$m(Q_k)$ depends upon $k$.%
\footnote{We can easily show that
$m(Q_k) = \left(\tfrac{1}{k-1}-\tfrac{1}{k}\right)m(Q_1)$
for $k>1$. Thus, $m(Q_1) = \biguplus_{k>1} m(Q_k)$.}
To ensure homogeneity, with bounded intervals of equal length having equal magnums,
we are led to introduce a reordering of $\Q$ such that the magnum of the
unit rational interval $Q_k = (k-1, k]_{\Q}$ is independent of $k$.

We now define a new ordering of $\Q$, denoting the reordered set as $\QB$.
We recall from Theorem~\ref{th:R-and-A} that reordering does not change the
value of the magnum, so that (\ref{eq:QplusMag}) and (\ref{eq:QMag}) remain valid.

\begin{figure}[h]
\begin{center}
\includegraphics[width=0.75\textwidth]{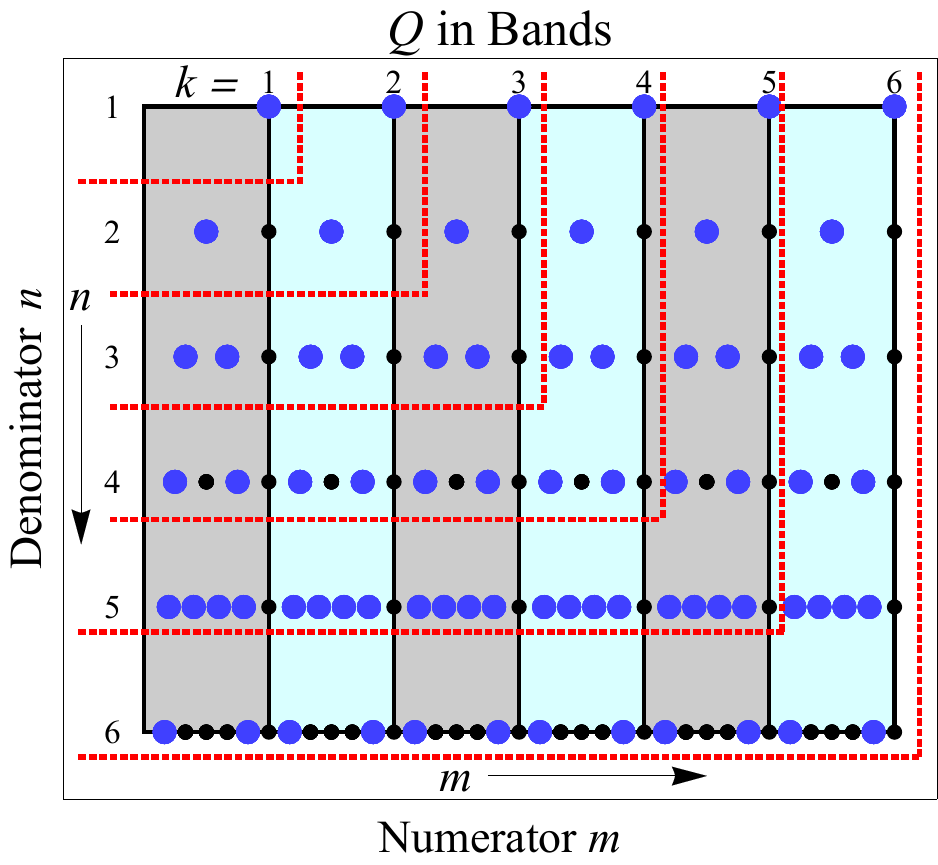}
\caption{Arranging $\Q$ in vertical bands of unit width.
Large points (blue online) represent rationals (with $\text{gcd}(m,n)=1$),
small points (black) are ordered pairs with $\text{gcd}(m,n)>1$, and $k$ is
the integer part of $m/n$.  Windows are separated by dashed (red) lines.
The numbers in $Q_k$ are in the $k$-th vertical band.}
\label{fig:Q-Bands}
\end{center}
\end{figure}

We begin by defining two families of subsets of $\N^2$:
$$
F(k,\ell) =  \left\{(m,\ell) : m = 1, 2, \dots, k\ell \right\}
\qquad\text{and}\qquad G(n) = \biguplus_{\ell=1}^n F(n,\ell) .
$$
Clearly, $|F(k,\ell)| = k\ell$ and $F(k_1,\ell)\subset F(k_2,\ell)$ if $k_1<k_2$.
Also, $|G(n)| = (n^3+n^2)/2$.


Confining attention to the rational numbers, we define
$$
H(n) = G(n)\cap \Q .
$$
It is clear that $H(n_1)\subset H(n_2)$ if $n_1<n_2$ and $|H(n)| = n\Phi(n) \approx \chi n^3$.
We order the $H$-sets according to the index $n$.
Within each set $H(n)$, we order the elements $(k,\ell)$ according to
increasing $\ell$ and, for each $\ell$, according to increasing $k$.
Thus, the ordering of the set $\QB$ begins as follows:
$$
\QB = \{
\textstyle{ \frac{1}{1} \ |\  
\frac{2}{1};\ 
\frac{1}{2},              \frac{3}{2}              \ |\ 
\frac{3}{1}, \frac{5}{2}            ;\ 
\frac{1}{3}, \frac{2}{3},             
  \frac{4}{3}, \frac{5}{3},              
  \frac{7}{3}, \frac{8}{3}              \ |\ 
\frac{4}{1},
\frac{7}{2}, \frac{10}{3}, \frac{11}{3}     ;\
\frac{1}{4}, \frac{3}{4}, \frac{5}{4}, \frac{7}{4}, \frac{9}{4},
\frac{11}{4}, \frac{13}{4}, \frac{15}{4}        \ |\
\frac{5}{1}, \dots }
\}
$$
We introduce a family $\mathscr{W}$ of windows: $W_1 = \{(1,1)\}$ and  $W_n = H(n)\setminus H(n-1)$ for $n>1$.
Then $|W_n| = n\varphi(n)$ and $\Lambda(n) = n\Phi(n)$.

We partition $\Q$ as a disjoint union of \emph{bands} of unit width,
$\Q = \biguplus_{k\in\N} Q_k$, where the $k$-th band is defined as
$$
Q_k = \{ (m,n) \in\Q^{+} : ( (k-1) < \tfrac{m}{n} \le k ) \} .
$$
The sets $Q_k$ are depicted in Fig.~\ref{fig:Q-Bands} as the vertical bands of unit width.
The extents of the sets $H(n)$ are indicated in the figure by the dashed lines (red online).
It may be seen that, when $n \ge \max(k_1,k_2)$, both $Q_{k_1}\cap H(n)$ and $Q_{k_2}\cap H(n)$
have $\Phi(n)$ elements. Therefore, the counting sequences 
$K_{Q_{k_1}}(n)$ and $K_{Q_{k_2}}(n)$ are equal. Indeed, we have the stronger result
$K_{Q_{k_1}}(n) \Eeq K_{Q_{k_2}}(n)$, from which it follows that
$m(Q_{k_1}) = m(Q_{k_2})$. Therefore, the magnums of \emph{all the vertical bands} $Q_k$ 
are equal: the magnum is homogeneous under the ordering $\QB$ and $m(Q_{k})$ is
independent of $k$.  We will now determine $m(Q_{1})$.



The $n$-th window of $\mathscr{W}$ is $H(n)\setminus H(n-1)$
and the union of the first $n$ windows is $H(n)$.
The set $H(n)$ contains $\Lambda(n)=n\Phi(n)$ elements.
Recall from (\ref{eq:QplusMag}) that $m_{\N^2}(\Q^{+}) \approx 2\chi\omega^2$.
If $\mathscr{W}$ were a fenestration of $\Q$, there would exist $\nu\in\Nn$
such that $\hat{\Lambda}(\nu) = m_{\N^2}(\Q)$, that is, $\nu\Phi(\nu) = 2\chi\omega^2$. 
From Hardy \&\ Wright \cite{HaWr60}, we have $\Phi(n) \approx \chi n^2$.
We may assume that this approximation extends to $\Nn$, so that 
$\nu\hat{\Phi}(\nu) \approx \chi\nu^3$ and
$\nu = \Lambdahatinv(2\chi\omega^2) \approx 2^{1/3}\omega^{2/3}$.

The counting sequence of $Q_1$ is $K_{Q_1}(n) = \Phi(n)$.
Using the approximations for $\Phi$ and for $\nu$, we get
\begin{equation}
m_{\QB}(Q_1) = \hat{K_{Q_1}}(\nu) \approx \hat\Phi(2^{1/3}\omega^{2/3}) = 2^{2/3}\chi\,\omega^{4/3} .
\label{eq:MagQ1}
\end{equation}
Thus, the magnum of each band $Q_k$ is of order $O(\omega^{4/3})$,
sitting comfortably between the magnums of $\N$ and $\N^2$.


\subsubsection*{Magnum of $\N$ Relative to $\Q$ and to $\QB$}


We consider the magnum of $\N$ relative to the rational numbers with the
two distinct orderings, square and banded. With square ordering, there is
a single natural number in each window, so $K_{\N}(n) = n$.
The fenestration has $\Lambda(n) = n^2$, so $\Lambda^{-1}(n) = \sqrt{n}$.
Thus
$$
m_{\Q}(\N) = \widehat{K_{\N}}(\Lambdahatinv(\omega^2)) = \omega .
$$
With banded ordering $\QB$, the relevant partition of $\Q$ has endpoints
$\Lambda(n) = n\Phi(n)$.
The counting function of $\N$ is $K_{\N}(n) = n$. 
The dashed (red) lines in Fig.~\ref{fig:Q-Bands} separate the windows $W_k$.
Since $\{\Lambda(k) : k\in\N\}$ is \emph{not} an omega set, the family
$\mathscr{W} = \{ W_n : n\in\N \}$ is not a fenestration, so we cannot use
the General Isobary Theorem. We compute the magnum of $\N$ using the approximate
solution $\nu \approx 2^{1/3}\omega^{2/3}$ obtained above:
$$
m_{\QB}(\N) = \hat{K_{\N}}(\Lambdahatinv(2\chi\omega^2)) \approx 2^{1/3}\omega^{2/3}.
$$
Thus, relative to $\QB$, the set of natural numbers has magnum of order $O(\omega^{2/3})$.
While this result may appear surprising, it is consistent with
the dependence of the magnum on the choice of reference set.



\section{Application of Theorems to Selected Sets}
\label{sec:Applics}

Using Theorem~\ref{th:16.24} together with equations (\ref{eq:MagDef}) and (\ref{eq:mra}),
we can determine the magnums of a wide range of sets of real numbers.
Let $A = r\N = \{rn:n\in\N\}$ where $r\in\R$ and $r\not=0$. The 
defining sequence is $(a_n) = rn$ so that $a^{-1}_n = n/r$ and 
$\kappa_A(n) = \lfloor n/r \rfloor$. Therefore,
$$
m(r\N) = \hat{\kappa_A}(\omega) = \frac{\omega}{r} .
$$
In a similar manner we show that if $A = \N+r = \{n+r:n\in\N\}$, then 
$$
m(\N+r) = \lfloor \omega - r \rfloor = \omega + \lfloor - r \rfloor .
$$
Some cases of particular interest include
$m(\N-\half) = \omega$, $m(\N+\half) = \omega-1$ and $m(\half\N) = 2\omega$.
 

Magnums for many other sets can be found. For example, consider the general
arithmetic sequence $A = \{\alpha(n):n\in\mathbb{N}\}$ where $\alpha(n) = kn+\ell$
for arbitrary $k\in\mathbb{N}$ and $\ell > -k$.  The inverse function
is $\alpha^{-1}(n) = (n-\ell)/k$ so that 
$\kappa_A(n) = \lfloor \alpha^{-1}(n) \rfloor$ extends to
\begin{equation}
\widehat{\kappa_A}(\omega) =  \biggl\lfloor \frac{\omega}{k} - \frac{\ell}{k} \biggr\rfloor
                                =  \frac{\omega}{k} + O(1) \,.
\label{eq:arithseq}
\end{equation}
The even numbers arise from $k=2$ and $\ell=0$, and (\ref{eq:arithseq})
yields $m(2\mathbb{N}) = \omega/2$. The odd numbers arise from 
$k=2$ and $\ell=-1$, and (\ref{eq:arithseq}) yields $m(2\mathbb{N}-1)=\omega/2$.
The odd numbers greater than $1$ arise from $k=2$ and $\ell=+1$, and
(\ref{eq:arithseq}) yields $m(2\mathbb{N}+1)=\omega/2-1$.%
\footnote{Note that the floor function and negation do not commute: for example, 
$ - \lfloor \frac{1}{2} \rfloor = 0$ whereas $\lfloor - \frac{1}{2} \rfloor = -1$.}

Now consider the general geometric sequence
$A = \{a\,r^{n-1}\}$ for arbitrary $a$ and $r$ in $\mathbb{N}$. 
The inverse of $\alpha(n)$ is $\alpha^{-1}(n) = \log(rn/a)/\log r = \log_{r}(rn/a)$,
so that $\kappa_A(n) = \lfloor \alpha^{-1}(n) \rfloor$ extends to
\begin{equation}
\widehat{\kappa_A}(\omega) =  \lfloor {\log_{r}({r\omega}/{a})} \rfloor 
                 =  \lfloor {\log_{r}\omega} \rfloor + O(1) \,.
\label{eq:geomseq}
\end{equation}
Thus, if $a=1$ and $r=2$ we have $A=\{1,2,4,8, \dots \}$ and $m(A)=\log_{2} 2\omega = \log_{2}\omega + 1$.

Magnums can be assigned to combinations of sets for which the definition in terms
of natural density fails. For example, using Theorem~\ref{th:Theorem-M2},
the set $A = 3\mathbb{N}\cup 4\mathbb{N}$ shown in Fig.~\ref{fig:AlphaKappa} has magnum 
$$
m(3\mathbb{N}\cup 4\mathbb{N}) = m(3\mathbb{N}) + m(4\mathbb{N}) - m(12\mathbb{N}) =
\frac{\omega}{3}+\frac{\omega}{4}-\frac{\omega}{12} = \frac{\omega}{2} \,.
$$
This also follows from the observation that
$3\mathbb{N}\cup 4\mathbb{N} \isobar 2\mathbb{N}$ with windows of length $12$.

For the set of perfect squares, $\N^{(2)} := \{ a(n)=n^2 : n\in\mathbb{N}\}$,
we consider the fenestration $\mathscr{W}$ determined by $\Lambda(n) = n^2$
(noting that $\Lambda\in\BOM$).
Clearly, there is one element in each window.
Since ${a}^{-1}(n) = \sqrt{n}$, we find that
$m(\N^{(2)}) = \lfloor \hat{a^{-1}}(\omega) \rfloor = \sqrt{\omega}$.

Now consider $B := \{(n^2-n+1) : n\in\mathbb{N}\}$.
Since there is one element in each window, $B$ is isobaric to $\N^{(2)}$ under $\mathscr{W}$
and so $m(B) = \sqrt{\omega}$.
More generally, for $\Lambda(n) = n^k$, $\Lambdahatinv(\omega) = \sqrt[k]{\omega} \in \Nn$.

The set $A = \mathbb{N}^{(2)}\cup \mathbb{N}^{(3)}$, containing all squares and cubes,
has magnum 
$$
m(\mathbb{N}^{(2)}\cup \mathbb{N}^{(3)}) =
m(\mathbb{N}^{(2)})+m(\mathbb{N}^{(3)})-m(\mathbb{N}^{(6)}) =
\sqrt[2]{\omega}+\sqrt[3]{\omega}-\sqrt[6]{\omega} \,.
$$


Combinations are possible, \emph{e.g.}, suppose
$A = \{ \alpha n^2 + \beta n + \gamma : n\in\N \}$ with $\alpha, \beta, \gamma$ in $\mathbb{N}$.
Then
$$
\hat{a}^{-1}(\omega) =
\sqrt{\frac{1}{\alpha}\left({\omega}+\frac{\Delta}{4\alpha^2}\right)} - \frac{\beta}{2\alpha}
$$
where $\Delta = \beta^2-4\alpha\gamma$. If $\Delta = 0$ and $2\alpha\, |\, \beta$,
this is in $\Nn$.  So, for example, if $\beta = 2k\alpha$, then
$$
a(n) = \alpha (n+k)^2 \qquad\text{and}\qquad
\hat{a}^{-1}(\omega) = \left(\sqrt{\frac{\omega}{\alpha}} - k\right) \in \Nn \,.
$$
For the prime numbers $\bbP = \{p_n:n\in\N\}$, the counting sequence is given by
the prime counting function, $\kappa_{\bbP}(n) = \pi(n)$. The Prime Number Theorem
assures us that $\pi(n) \approx n/\log n$. Thus, using Theorem~\ref{th:DensityTheorem},
the Density Theorem, we can estimate the magnum of the primes to be
$m(\bbP) \approx \omega/\log\omega$.

\begin{center}
\begin{table}[h]
\caption{Magnums of selected subsets of $\mathbb{N}$
         computed from $m(A)=\hat{\kappa_A}(\omega)$.}
\label{tab:SampleMagnums}
\begin{tabular}{ c   c   c   c   c }   
\hline
 Set                &  Definition    &  Inverse      &   Counter                     &   Magnum      \\
 $A$                &  $a(n)$        & $a^{-1}(n)$   &  $\kappa(n)$                  &   $m(A)$      \\
\hline
$\mathbb{N}$        &     $n$        &    $n$        &   $\lfloor n \rfloor = n$     &  $\omega$     \\
$\mathbb{N}\setminus\{1\}$  &     $n+1$    &    $n-1$    &   $\lfloor n-1 \rfloor$  &  $\omega-1$     \\
$2\mathbb{N}$       &     $2n$       &    $n/2$      &   $\lfloor n/2 \rfloor$       &  $\omega/2$   \\
$2\mathbb{N}-1$     &     $2n-1$     & $(n+1)/2$     & $\lfloor (n+1)/2 \rfloor$     &  $\omega/2$   \\
$2\mathbb{N}+1$     &     $2n+1$     & $(n-1)/2$     & $\lfloor (n-1)/2 \rfloor$     &  $\omega/2-1$ \\
$k\mathbb{N}$       &     $kn$       & $n/k$         & $\lfloor n/k \rfloor$         &  $\omega/k$   \\
$k\mathbb{N}+m$     &     $kn+m$     & $(n-m)/k$     & $\lfloor (n-m)/k \rfloor$     &  $\omega/k+\lfloor -m/k\rfloor$   \\
Perfect Squares     &     $n^2$      & $\sqrt{n}$    & $\lfloor \sqrt{n} \rfloor$    &  $\sqrt{\omega}$   \\
Triangular Numbers  &  $(n^2+n)/2$    & $\sqrt{2n+\frac{1}{4}}-\frac{1}{2}$    & $\left\lfloor \sqrt{2n+\frac{1}{4}}-\frac{1}{2} \right\rfloor$    &  $\sqrt{2\omega}-1$   \\
$k$-th Powers       &     $n^k$      & $\sqrt[k]{n}$ & $\lfloor \sqrt[k]{n} \rfloor$ &  $\sqrt[k]{\omega}$\\
Geometric Series  &  $k^n$    & $\log_k n$    & $\left\lfloor \log_k n \right\rfloor$    &  $\lfloor \log_k \omega \rfloor$   \\
Fibonacci Numbers &  $\lfloor\varphi^n/\sqrt{5}\rceil$    & $\lfloor\log_\varphi(\sqrt{5}n)\rceil$    & $\left\lfloor \log_\varphi(\sqrt{5}n) \right\rfloor$    &  $\lfloor \log_\varphi(\sqrt{5}\omega) \rfloor$   \\
Prime Numbers       &   $p_n$    &   $\pi(n)$    &   $\pi(n)$   & $\approx\lfloor \omega/\log\omega \rfloor$  \\
\hline
\end{tabular}   
\end{table}
\end{center}


Many more examples could be presented. In Table~\ref{tab:SampleMagnums} 
we list the magnums for a range of subsets of $\mathbb{N}$.


We have obtained the magnums of many subsets of $\mathbb{N}$, but
it must be noted that there are many more sets for which we are unable
to calculate the magnums.  As an example, consider the set
$\mathbf{Od}_{2}$, whose elements are all natural numbers having an
odd number of binary digits.  The natural density of this set
oscillates between values that asymptote to $\frac{1}{3}$ and
$\frac{2}{3}$, never tending to a limit (see Fig.~\ref{fig:BinDigs}).
Moreover, the distance
between one extreme and the next increases by a factor of $2$ for
each stage. It seems possible that some additional assumption or
axiom is required to determine the magnum of $\mathbf{Od}_{2}$ uniquely.
%
\begin{figure}[h]
\begin{center}
\includegraphics[width=0.75\linewidth]{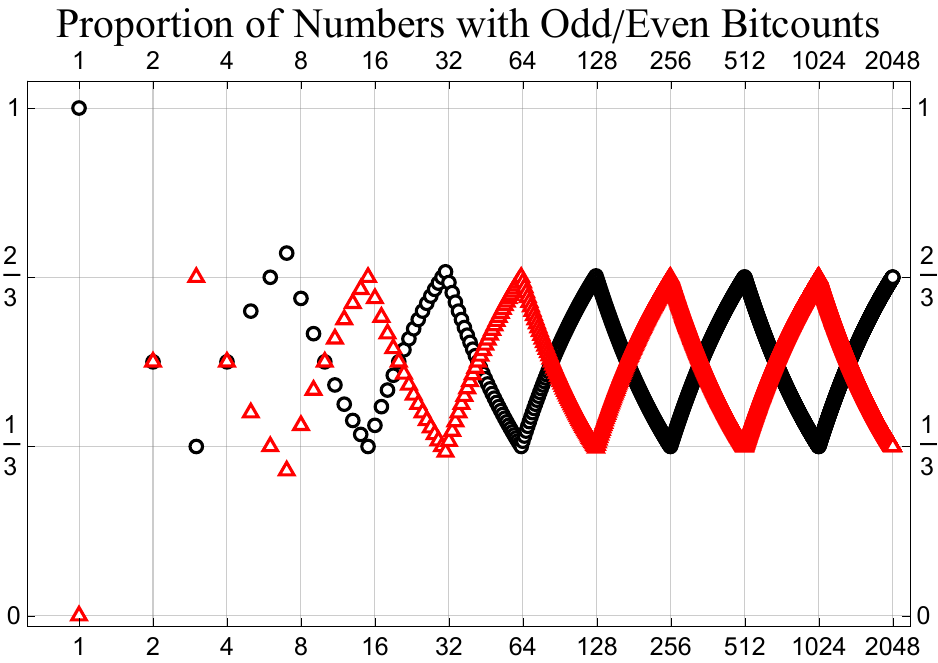}
\caption{Proportion of numbers less than $n$ that have
an odd count of binary digits (open circles) and 
an even digit count (triangles, red online).
Note that the $x$-axis is logarithmic.}
\label{fig:BinDigs}
\end{center}
\end{figure}


\section{Conclusions and Prospects}
\label{sec:Conclusions}

We have defined a function --- the magnum --- that enables us to quantify
the sizes of a wide range countable sets in terms of surnatural numbers.
The magnum of a set $A$ may be constructed genetically or defined by extending
the domain of the counting function. Magnums respect the Euclidean Principle,
so that a proper subset of a set is smaller than that of the set itself.

We introduced the concept of isobary and proved that isobarically equivalent
sets have equal magnums. The magnums of some sets of natural numbers
were obtained and, in addition, a range of larger sets. With the canonical square
ordering of $\N^2$, and using a basic result of number theory,
the leading term of the magnum of the rational numbers was shown to be
$m(\Q^{+}) = (6/\pi^2)\omega^2$.
Crucially, this value is independent of the ordering of $\Q$.
With banded ordering of $\Q$, the set $Q_k$ of rational numbers in the interval
$(k-1,k]$ was shown to be $m_{\Q_{B}}(Q_k) = O(\omega^{4/3})$ and independent of $k$.


A fundamental feature of the magnum $m(A)$ of a set $A$ is its dependence
on the ordering of the set $R$ chosen as reference set. In this sense,
all magnums are `relative'.  This is perhaps not surprising when we recall
that an infinite set may have more than one ordinal assigned to it.

Several important issues remain to be clarified and there is ample scope for 
extension of the theory of magnums.  This theory provides many opportunities
for development and application: counting with surnatural numbers
is a rich source of interesting open problems. We conclude with a list of
some important outstanding questions:
%
%
\begin{enumerate}
\item The Axiom of Extension (pg.~\pageref{ax:exten}) posits that every nondecreasing
function $f:\N\to\N_0$ can be extended to $\hat f:\Nn\to\Nn$ while conserving specific properties.
Can this axiom be derived from a more fundamental assumption that still guarantees that
$\mathscr{M}$, the family of sets for which magnums can be defined, contains every countable set? 
\item From the structure of the surreal number tree, there is symmetry under
reflection about zero, from which it follows that $m(-A) = m(A)$.
Are there any other symmetries?
\item Can the family $\BOM$, which forms a free ultrafilter, be used in a manner
similar to the approach employed in non-standard analysis?
\end{enumerate}

\section*{Acknowledgements}


\begin{small}
We are grateful to
Prof Philip Ehrlich, Ohio University, 
Dr Vincent Bagayoko, Universit\'{e} Paris Cit\'{e}
and
Dr Vincenzo Mantova, University of Leeds for helpful comments on our paper
\cite{LyMa23} on \textsc{arXiv.org},
which helped us in the preparation of the present paper.

The research of Benci \emph{et al.}~on numerosities described in \cite{BeNa19}
has inspired several strands of investigation in our development of
magnum theory. The paper of Trlifajov\'{a} \cite{Trl24} has also
been of great assistance.
\end{small}



\addcontentsline{toc}{section}
{$\bullet$ \ Appendix 1: Surreal Numbers and Surnatural Numbers}
\section*{Appendix 1: Surreal Numbers and Surnatural Numbers}
\label{sec:appendix1}



The class of surreal numbers, denoted by $\No$,
discovered by Conway around 1970, is the largest possible ordered field
(Conway, 2001, referenced as \cite{ONAG}).
The basic arithmetical operations ---
addition, subtraction, multiplication, division and roots ---
are defined in $\No$.

The surreal numbers are constructed inductively. 
We give a brief sketch, without proofs or details, of some of their
key properties.
A full account is given in \cite{ONAG}, an elementary
discussion in \cite{Knu74} and a tutorial in \cite{Sim17}.

Every number $x$ is defined as a pair of sets of numbers,
the left set and the right set:
$$
x = \{ X^{L} \ |\  X^{R} \}
$$
where every element of $X^{L}$ is less than every element of $X^{R}$, and
$x$ is the \emph{simplest} or earliest appearing number between these two sets.
We start with $0$, defined as $\{\varnothing |\varnothing \}$. Then 
$$
\{0\ |\varnothing \} = 1 \qquad \{1\ |\varnothing \} = 2 \qquad \{2\ |\varnothing \} = 3 \qquad \dots
$$
Negative numbers are defined inductively as $-x=\{-X^{R}\ |\ -X^{L}\}$. So
$$
\{\varnothing |\ 0\} = -1 \qquad \{\varnothing | -1\} = -2 \qquad \{\varnothing | -2\} = -3 \qquad \dots
$$
The dyadic fractions (rationals of the form $m/2^n$) appear as
$$
\{0\ |\ 1\} = \textstyle{\frac{1}{2}} \qquad
\{1\ |\ 2\} = \textstyle{\frac{3}{2}} \qquad
\{0\ |\ \half\} = \textstyle{\frac{1}{4}} \qquad
\{\half\ |\ 1\} = \textstyle{\frac{3}{4}} \qquad
\dots
$$
After an infinite number of stages, when all dyadic fractions have emerged,
all the remaining rationals and real numbers appear.
The first infinite and infinitesimal numbers are defined at this stage:
$$
\omega = \{0, 1, 2, 3, \dots\ \ |\varnothing \} \,, \qquad 
1/\omega = \{\ 0 \ |\ 1, \textstyle{\frac{1}{2}, \frac{1}{4}, \frac{1}{8}}, \dots\ \} \,.
$$
We can then introduce
$$
\omega + 1 = \{0, 1, 2, 3, \dots,\ \omega\ |\varnothing \} \qquad
\omega - 1 = \{0, 1, 2, 3, \dots\ \ |\ \ \omega\} 
$$
$$
2\omega = \{0, 1, 2, 3, \dots\ \omega,\omega+1,\omega+2, \dots \ |\varnothing \} \qquad
\half\omega = \{0, 1, 2, 3, \dots\ \ |\ \ \omega,\omega-1,\omega-2, \dots \}
$$
and many other more exotic numbers. They behave beautifully:
$\No$ is a totally ordered field, indeed, the largest such field.

The sum of two surreal numbers $x_1 = \{ X^L_1 \ |\ X^R_1 \}$ and 
$x_2 = \{ X^L_2 \ |\ X^R_2 \}$ is defined to be
$$
x_1+x_2 = \{ x_1 + X^L_2, X^L_1 + x_2 \ |\ x_1 + X^R_2, X^R_1 + x_2 \ \}
$$
where, for example, $x_1 + X^L_2$ is the set of numbers $x_1 + x^L_2$
for all elements $x^L_2$ of $X^L_2$.

The product of two surreal numbers is defined to be
\begin{eqnarray*}
x_1 x_2 &=& \{ x_1^L x_2 + x_1 x_2^L - x_1^L x_2^L, x_1^R x_2 + x_1 x_2^R - x_1^R x_2^R \ |\ \\
    && \ \   x_1^L x_2 + x_1 x_2^R - x_1^L x_2^R, x_1^R x_2 + x_1 x_2^L - x_1^R x_2^L \} \,.
\end{eqnarray*}

\subsubsection*{The Omnific Integers}

Conway (\cite{ONAG}, Ch.~5) defines the class $\Oz$ of \emph{omnific integers}:
$x$ is an omnific integer if $x = \{x-1\ |\ x+1\}$. Omnifics are the appropriate
class of integers in the context of the surreals. They include all the real
integers and all ordinal numbers, and every surreal number $x$ is the quotient
of two omnific integers.

Any surreal number $x$, may be expressed uniquely in \emph{normal form}, 
\begin{equation}
x = \sum_{\beta < \alpha} r_\beta \cdot \omega^{y_\beta} \,,
\label{eq:normalform}
\end{equation}
where $\alpha$ is an ordinal, the numbers $r_\beta$ are real
and the numbers $y_\beta$ form a descending sequence.
The number expressed by (\ref{eq:normalform}) is an omnific integer if and only if
$r_\beta = 0$ for $y_\beta < 0$ and $r_0 \in \mathbb{Z}$ 
\cite[Ch.~5, Th.~1]{ONAG}.

There are no infinitesimals in $\Oz$.
Every surreal number is distant at most $1$ from
some omnific integer, and $\No$ is the fraction field of $\Oz$.
For surreal arguments $s$ we define $\lfloor s \rfloor$ to be the (omnific)
integer part of $s$. If the normal form of $s$ is
$s = \sum_{\beta<\alpha} \omega^{y_\beta}\cdot r_\beta$ then
$ \lfloor s \rfloor = \sum_{0 < \beta<\alpha} \omega^{y_\beta}\cdot r_\beta + 
\lfloor r_0 \rfloor $\,. 

For an illuminating discussion of the omnific integers, see \cite[\S7]{Sim17}.

\subsubsection*{The Surnatural Numbers}

If we confine attention to subsets of $\mathbb{N}$,
we need nothing larger than $\omega$ to define their magnums.
Moreover, we assume that the magnum function maps the
power set of the natural numbers into the positive omnific integers
or \emph{surnatural numbers}, $\Nn = \Oz^{+}\cup\{0\}$:
\begin{equation}
m : \mathscr{P}(\mathbb{N}) \to \Nn \,.
\label{eq:map2surnat}
\end{equation}
This rules out negative and fractional values for $m(A)$, but admits numbers
$r\cdot\omega^y$, where $r$ is an arbitrary positive real number and $y\in(0,1]$.
For example, $\omega/3$ and $\omega^{1/2}$ are admissible.

To be explicit, a surnatural number may be expressed in the normal form
(\ref{eq:normalform}) where $\alpha$ is an ordinal, the numbers $r_\beta$
are non-negative real numbers, the numbers $y_\beta$ form a descending
sequence of non-negative reals, and $r_0$ is a non-negative integer.
Numbers with $r_0=0$ are called \emph{purely infinite} surnatural numbers.

It is noteworthy that, since all numbers $r\cdot\omega^\beta$
with $r\in\R$ and $\beta>0$ are surnaturals,
it follows that $\omega$ is divisible by all numbers in $\mathbb{N}$,
and indeed in $\mathbb{R}$.
Thus, $\omega$ is an even number, since $\omega/2\in\Nn$, a multiple
of $3$ since $\omega/3\in\Nn$, and so on. Indeed, for all $k\in\mathbb{N}$, $k!$ divides $\omega$.
Moreover, $\sqrt[k]{\omega}\in\Nn$, so $\omega$ is a perfect square,
a perfect cube, and so on.

\subsubsection*{Surreal Numbers \&\ Current Research}

Since their invention/discovery around 1970, surreal numbers have attracted
relatively little attention. In \emph{The Princeton Companion to Mathematics}
\cite{PCM08},
a comprehensive review of mathematics extending to more than 1{,}000 pages,
there is no index entry for `surreal numbers'. J.H.Conway's name is indexed,
but all four references are to his work in group theory.

Recently, Bertram \cite{Bert25} has advanced a new formulation
of the theory of surreal numbers on the foundation of standard set
theory.  He constructs a hierarchy having properties similar to the
von Neumann hierarchy $\VN$, but with two kinds of membership, $\in_L$
(left member) and $\in_R$ (right member) instead of just one type
$\in$.  The bibliography in \cite{Bert25} includes a good
selection of recent publications on surreal numbers.

Hamkins \cite{Ham20} refers repeatedly in his \emph{Lectures on the Philosophy of Mathematics}
to the Field of surreal numbers.  He describes Conway's creation of surreal numbers,
as expounded in \cite{ONAG}, as ``a beautiful and remarkable mathematical theory.''
He notes \cite[pg.~79]{Ham20} a categoricity result for the surreal numbers,
which form a nonstandard ordered field of proper class size: the class of surreals is unique
in the sense that, in G\"{o}del-Bernays set theory with global choice,
all such ordered fields are isomorphic.

In the \emph{Notices of the American Mathematical Society}
from 2018 to 2023 \cite{NAMS}, thesis titles were listed for a total of
about 8{,}000 doctoral degrees awarded in the mathematical sciences
between July 1,~2016 and June 30,~2020,
as reported by about 250 departments in 200 universities in the United States.
There were no occurrences of the word `surreal' in these lists.

Although Conway regarded surreals as his most significant creation,
they have attracted relatively little attention; given their appealing
properties and their remarkable elegance, this is surprising. 
A small number of researchers have worked on surreal numbers, and
they have also been of interest to recreational mathematicians.
For a review of progress over the past half-century, see \cite{Eh22}
in \cite{Ryba22}.

\addcontentsline{toc}{section}
{$\bullet$ \ Appendix 2: Calendar of Magnums for Selected Sets}
\section*{Appendix 2: Calendar of Magnums for Selected Sets}
\label{sec:appendix2}

When is the magnum of a subset of $\mathbb{N}$ first defined?
We consider the ordinals as they arise.

\textbf{Day 0:} The magnum of $\varnothing$ is defined to be $0$.

\textbf{Day 1:} The magnum of the singleton $\{n\}$ is defined
for all $n\in\mathbb{N}$ to be $1$.

\textbf{Day 2:} The magnum of the doubleton $\{m, n\}$ is defined
for all $m, n\in\mathbb{N}$ to be $2$.

\textbf{Day n:} The magnums of all sets with $n$ elements
are defined to be $n$.

Thus, all finite subsets of $\mathbb{N}$ are defined on finite days.
Their magnums are all the finite ordinal numbers.

\textbf{Day $\omega$:} The set $\mathbb{N}$ is given a magnum
on this day: $m(\mathbb{N}) = \omega$, the first infinite ordinal.
From now on, both ordinals and other surnatural numbers will be used.

\textbf{Day $\omega+1$:} 
All ``co-singletons'', sets of the form $\mathbb{N}\setminus\{k\}$,
are assigned the value $\omega-1$. Sets $\N\sqcup\{k\}$ are assigned $\omega+1$.

\textbf{Day $\omega+n$:}
All sets with complements having magnum $n$ are assigned the value $\omega-n$.
Sets $\N\sqcup\{1,2,\dots,n\}$ are assigned $\omega+n$.

Thus, before day $2\omega$, all finite and cofinite sets, together
with $\mathbb{N}$, have magnums assigned.

\textbf{Day $2\omega$:} The new magnum is $\omega/2$.
The obvious candidate for this magnum is $2\mathbb{N}$.
In fact, non-denumerably many other sets have  magnums assigned on this day.
For any natural number $L$, we consider all sets isobaric to
$2\mathbb{N}$ for a fenestration of $\mathbb{N}$
having window length $L$. Even for $L=2$, we have
an uncountable collection of sets isobaric to $2\mathbb{N}$.

The question is: \emph{have we now got {all} the sets}
with magnum $\omega/2$? 
To define the magnum of a candidate set $A$ that is neither finite nor cofinite,
we consider the magnum form:
\begin{equation}
m(A) = \{ m(B) : B\in\Omag_{2\omega} \land B \subset A \ |
        \ m(C) : C\in\Omag_{2\omega} \land C \supset A \} \,,
\nonumber
\end{equation}
The only (old) magnums available
on day $2\omega$ are $\{m : m\in\mathbb{N} \}$ and $\{\omega-n : n\in\mathbb{N}_0 \}$.
Thus, all the admissible sets $B\subset A$ are finite and all the admissible sets
$C\supset A$ are cofinite.  The set $A$ has infinite subsets and supersets,
but the magnums of these are not yet assigned.

The set $2\mathbb{N}$ certainly has subsets and supersets whose magnums have not yet appeared.
For example, $4\mathbb{N}\subset 2\mathbb{N}$ and $2\mathbb{N}\uplus(4\mathbb{N}+1) \supset 2\mathbb{N}$.
We can rule out their membership of $\Nmag_{2\omega}$ since, by the Euclidean Principle,
proper subsets and proper supersets of a set cannot have magnums equal to it.

A calendar of some subsets of $\mathbb{N}$, and of some larger sets, that acquire magnums
up to day $\omega^2$ is presented in Table~\ref{tab:timtab}.



\begin{table}[h]
\caption{Calendar of days on which some simple sets are given magnums.
$j$, $k$, $\ell$, $m$ and $n$ are natural numbers.}
\label{tab:timtab}

      \begin{small}

\begin{center}
\begin{tabular}{||c|c|l|c|l||}   
\hline\hline   
Day \hash   &     Magnum   &   Sets in $\mathscr{P}(\N)$  &  Magnum  &  Larger Sets   \\
\hline\hline   
   $0$      &      $0$     &   $\varnothing$    &    &   \\
\hline   
   $1$      &      $1$     &   $\{ k \}$   &    &   \\
\hline   
   $2$      &      $2$     &   $\{ k,\ell \}$    &    &   \\
\hline   
   $3$      &      $3$     &   $\{ k,\ell, m \}$   &    &   \\
\hline   
 $\cdots$   &   $\cdots$   &       $\cdots$        &    &   \\
\hline   
   $n$      &      $n$     &   $\{k_1 ,k_2, \dots , k_n \}$  &    &   \\
\hline   
 $\cdots$   &   $\cdots$   &       $\cdots$        &    &   \\
\hline   
 $\omega$   &   $\omega$   &       $\N$        &    &   \\
\hline   
 $\omega+1$ &   $\omega-1$ &    $\N\setminus\{k\}$  &  $\omega+1$   & $\N\sqcup\{k\}$   \\
\hline   
 $\omega+2$ &   $\omega-2$ &    $\N\setminus\{k,\ell\}$ & $\omega+2$   & $\N\sqcup\{k,\ell\}$  \\
\hline   
 $\cdots$   &   $\cdots$   &       $\cdots$        & $\cdots$   & $\cdots$  \\
\hline   
 $\omega+n$ &   $\omega-n$ &    $\N\setminus\{k_1, k_2, \dots ,k_n \}$ &  $\omega+n$ &  $\N\sqcup\{k_1, k_2, \dots ,k_n \}$ \\
\hline   
 $\cdots$   &   $\cdots$   &       $\cdots$        & $\cdots$   & $\cdots$  \\
\hline   
 $2\omega$ &   $\frac{\omega}{2}$     & $2\N$ & $ 2\omega$   & $\N^{\sqcup 2}:=\N\sqcup\N$  \\
           &                          & $2\N-1$ &            & $\half\N=\N\uplus(\N-\frac{1}{2})$  \\
\hline   
 $2\omega+1$ &   $\frac{\omega}{2}+1$ & $2\N\uplus\{2k-1\}$  & $2\omega+1$   & $\N^{\sqcup 2}\sqcup\{(k,\ell)\}$  \\
             &   $\ \ \frac{\omega}{2}-1$ &    $2\N\setminus\{2k\}$ & $2\omega-1$  & $\N^{\sqcup 2}\setminus\{(k,\ell)\}$ \\
\hline   
 $2\omega+2$ &   $\frac{\omega}{2}+2$ & $2\N\uplus\{2k-1,2\ell-1\}$  & $2\omega+2$  & $\N^{\sqcup 2}\sqcup\{(k_1,\ell_1),(k_2,\ell_2)\}$  \\
             &   $\ \ \frac{\omega}{2}-2$ & $2\N\setminus\{2k,2\ell\}$ & $2\omega-2$  & $\N^{\sqcup 2}\setminus\{(k_1,\ell_1),(k_2,\ell_2)\}$  \\
\hline   
 $\cdots$   &   $\cdots$   &       $\cdots$        & $\cdots$   & $\cdots$  \\
\hline   
 $3\omega$ &   $\frac{\omega}{4}$     &    $4\N$              & $3\omega$   &  $\N^{\sqcup 3}:=\N\sqcup\N\sqcup\N$  \\ 
           &   $\frac{3}{4}\omega$    &    $\N\setminus 4\N$  &             &  $\frac{1}{3}\N$       \\
\hline   
 $3\omega+1$ &   $\frac{\omega}{4}+1$     &    $4\N\uplus\{4k-1\}$   &  $3\omega+1$  &  $\N^{\sqcup 3}\sqcup\{(k,\ell,m)\}$ \\
             &   $\ \ \frac{\omega}{4}-1$     &    $4\N\setminus\{4k\}$  & $3\omega-1$    & $\N^{\sqcup 3}\setminus\{(k,\ell,m)\}$  \\
             &   $\frac{3}{4}\omega+1$    &    $(\N\setminus 4\N)\uplus\{4k\}$    &    &   \\
             &   $\ \ \frac{3}{4}\omega-1$    &    $(\N\setminus 4\N)\setminus\{4k-1\}$ &    &   \\
\hline   
 $\cdots$   &   $\cdots$   &       $\cdots$        & $\cdots$   & $\cdots$  \\
\hline   
 $4\omega$ &   $\frac{1}{8}\omega$     &    $8\N$  &  $4\omega$   &  $\N^{\sqcup 4}:=\N\sqcup\N\sqcup\N\sqcup\N$   \\
           &   $\ \ \frac{3}{8}\omega$     &    $8\N\uplus(8\N-1)\uplus(8\N-2)$  &    & $\frac{1}{4}\N$   \\
           &   $\frac{5}{8}\omega$     &    $\N\setminus[8\N\uplus(8\N-1)\uplus(8\N-2)]$   &    &   \\
           &   $\ \ \frac{7}{8}\omega$     &    $\N\setminus 8\N$    &    &   \\
\hline   
 $\cdots$   &   $\cdots$   &       $\cdots$        & $\cdots$   & $\cdots$  \\
\hline   
 $\omega^2$ &  $\frac{j}{k}\omega$ &   $\biguplus_{m=0}^{j-1} (k\N-m)$   & $\omega^2$   & $\N^2:=\N\times\N$  \\
           &           $\cdots$       &              $\cdots$           &    &   \\
           &       $\sqrt{\omega}$    & $\N^{(2)} = \{1,4,9,16,\dots\}$         &    &   \\
\hline\hline   
\end{tabular}   
\end{center}

      \end{small}

\end{table}
\end{document}